\documentclass[review]{elsarticle}

\usepackage{lineno,hyperref}
\modulolinenumbers[5]
\usepackage{multirow}
\usepackage{makerobust}
\MakeRobustCommand\overrightarrow
\MakeRobustCommand\overleftarrow
\usepackage{pdflscape}
\usepackage{endnotes}
\usepackage{amsmath}             
\usepackage{graphicx}
\usepackage{caption}
\usepackage{subcaption}
\usepackage{amsmath, amsthm, amssymb}
\usepackage{appendix}
\usepackage{graphicx}
\usepackage{resizegather}
\usepackage{color}
\usepackage{xcolor}
\usepackage{chngcntr}
\usepackage{apptools}
\usepackage{float}

\newtheorem{proposition}{Proposition}
\newtheorem{corollary}{Corollary}
\newtheorem{lemma}{Lemma}

\newtheorem{remark}{Remark}
\AtAppendix{\counterwithin{remark}{section}}
\AtAppendix{\counterwithin{proposition}{section}}

\begin{document}

\begin{frontmatter}

\title{\textcolor{black}{Controlling arrival and service rates to reduce sensitivity of queueing systems with customer abandonment}}


\author[mymainaddress]{Katsunobu Sasanuma\corref{mycorrespondingauthor}}
\ead{katsunobu.sasanuma@stonybrook.edu}
\cortext[mycorrespondingauthor]{Corresponding author}

\author[mysecondaryaddress]{Robert Hampshire}
\ead{hamp@umich.edu}

\author[mythirdaddress]{Alan Scheller-Wolf}
\ead{awolf@andrew.cmu.edu}

\address[mymainaddress]{College of Business, Stony Brook University, Stony Brook, NY 11794, USA}
\address[mysecondaryaddress]{Ford School of Public Policy, University of Michigan, MI 48109, USA}
\address[mythirdaddress]{Tepper School of Business, Carnegie Mellon University, Pittsburgh PA 15213, USA}


\begin{abstract}
\textcolor{black}{The Erlang A model--an M/M/s queue with exponential abandonment--is often used to represent a service system with impatient customers. For this system, the popular square-root staffing rule determines the necessary staffing level to achieve the desirable QED (quality-and-efficiency-driven) service regime; however, the rule also implies that properties of large systems are highly sensitive to parameters. We reveal that the origin of this high sensitivity is due to the operation of large systems at a point of singularity in a phase diagram of service regimes. We can avoid this singularity by implementing a \emph{congestion-based control} (CBC) scheme--a scheme that allows the system to change its arrival and service rates under congestion. We analyze a modified Erlang A model under the CBC scheme using a Markov chain decomposition method, derive non-asymptotic and asymptotic normal representations of performance indicators, and confirm that the CBC scheme makes large systems less sensitive than the original Erlang A model.}
\end{abstract}



%


\end{frontmatter}
\section{Introduction}
\textcolor{black}{When a service facility is congested, we often observe impatient customers; they may decide to abandon the facility and leave, either by balking (not joining a queue) or reneging (leaving a queue). To evaluate the quality of service (QoS) of such facilities, it is convenient to use the Erlang A queueing model--an M/M/s queueing model with exponential reneging. The analysis of this model has revealed that three distinct asymptotic regimes exist: Quality-and-Efficiency-Driven (QED), Quality-Driven (QD), and Efficiency-Driven (ED) regimes \cite{garnett2002designing,mandelbaum2005palm}. Among these regimes, QED is \emph{practically} important since its delay probability (which we denote as $P_Q$) achieves a value strictly between~0 and~1, balancing good service and reasonable capacity cost. This QED regime is realized following the square-root staffing rule, i.e., setting the number of staff~$s$ in the vicinity of the square root of the resource requirement~${R=\lambda/\mu}$, where~$\lambda$ is an arrival rate and $\mu$ is a service rate per worker. For example, for a large system with $R=10,000$, $s$ could be set between $R \pm \sqrt{R} = 9,900-10,100$, which suggests that the control range of $s$ is within $\pm 1\%$ of $R$. At the limit of large $R$, the control range of $s$, when normalized by $R$, approaches zero. Hence, a small fluctuation of $s$, $\lambda$, or $\mu$, and thus $s/R$ could make the system fall into either the ED ($P_Q \approx 1$) or QD ($P_Q \approx 0$) regimes. In fact, it has been pointed out that for large systems, ``the above three-regime dichotomy is rather delicate'' \cite{mandelbaum2005palm} and ``operating in the QED regime, the performance measures tend to be highly sensitive to changes in the arrival rate, the service rate, or the number of servers'' \cite{whitt2006sensitivity}. Thus, in order to maintain a system operating in a QED regime, facility operators need to constantly alter the staffing level in response to variations in the arrival and service rates; otherwise, the system could easily suffer from extreme delay probabilities--100\% or 0\%. Such a drastic fluctuation is obviously unfavorable for business. However, most previous literature focused on refining square-root staffing rules, and has not explained the origin of, or provided solutions to, this extreme sensitivity.}

\textcolor{black}{To reduce the sensitivity of the QoS for a system with customer abandonment, we propose a \emph{congestion-based-control} (CBC) scheme that controls the system's arrival and service rates during congestion. The implementation of the CBC scheme is commonly observed in practice: facility managers may try to temporarily suppress their arrival rate by informing customers a system is congested; they may also try to induce a faster service rate by providing staff with monetary incentives when a system is congested.}

\textcolor{black}{To evaluate the impact of the CBC scheme, we study a modified Erlang A model, in which the original Erlang A model is extended to incorporate either reneging or balking, and is allowed to change arrival and service rates during congestion. We provide both non-asymptotic and asymptotic normal representations of performance indicators for our modified Erlang A model. A non-asymptotic representation provides numerically accurate formulae to calculate performance indicators for even small systems, for which the asymptotic formulae are not reliable. An asymptotic representation provides insights into the sensitivity of the modified Erlang A system, i.e., operation at a point of singularity in a phase diagram, and explains a simple rule-of-thumb on how to avoid a singularity, expand the range of the \emph{desired} QED regime, and achieve stable operations of large systems. Specifically, we show that under the CBC scheme, the square-root staffing rule to realize the QED regime is converted to the linear staffing rule, making the control range of the number of servers to maintain the QED regime the order of $R$ instead of $\sqrt{R}$. Thus the system under the CBC scheme becomes more robust and less sensitive to parameters than the original Erlang A model.}


The remainder of this paper is organized as follows. Section~\ref{sec:litrev} reviews the related literature. Section~\ref{sec:model} explains the modified Erlang A model with the CBC scheme. In Section~\ref{sec:perform}, we take a Markov chain decomposition approach to represent performance indicators using blocking probabilities of decomposed sub-chains and obtain a non-asymptotic normal representation of performance indicators. In Section~\ref{sec:asymptotic}, we take a limit of large systems and derive phase diagrams of service regimes. We show the results of numerical experiments in Section~\ref{sec:numerical}. Finally, Section~\ref{sec:conclusion} concludes the paper.

\section{Literature Review}\label{sec:litrev}
Customer abandonment, such as reneging (leaving a queue while waiting) and balking (leaving a system before joining a queue), has been one of the main interests in queueing systems for a very long time. \cite{palm1957research} built a multi-server queueing model with reneging, the Erlang A (M/M/n+M) model, that allows customers to abandon a system. In this system, each arrival has exponential patience and reneges if the waiting time exceeds the patience. \cite{udagawa1957queue,haight1959queueing,ancker1963a} studied customer abandonment using a single server queueing model. They not only analyzed reneging, but also balking using the same framework and provided exact solutions \cite{ancker1963a}; however, their exact solutions are only numerically tractable.

In more recent years, researchers have been interested in obtaining intuition and simple rules-of-thumb that can be used by practitioners. For this purpose, they have developed various approximation techniques. One of the most frequently used methods is heavy-traffic approximation, which is sufficiently accurate when congestion is persistent. Researchers have found many simple rules-of-thumb for congested systems by applying heavy-traffic approximation. For example, heavy-traffic approximation has been applied to on-street parking problems \cite{larson2010congestion}, public housing applications \cite{kaplan1987analyzing}, and kidney transplantations \cite{zenios1999modeling}.

Heavy-traffic approximation is simple and effective, but it can only be applied to a system under heavy congestion. To study a system not in heavy congestion, \cite{halfin1981heavy} apply an asymptotic method (diffusion approximation) to analyze the Erlang C model (an M/M/s queue with no customer abandonment) and provide an important, yet simple, square-root staffing rule to identify the staffing level that satisfies a desired QoS level. \cite{garnett2002designing} apply the same asymptotic framework to the Erlang A model and classify its performance into three distinctive phases (regimes): QED (asymptotic limit of the probability of queueing $P_Q$ is strictly between 0 and 1), QD ($P_Q \to 0$), and ED ($P_Q \to 1$).

\textcolor{black}{The square-root staffing rule and its more refined variations have been well-studied \cite[see, for example,][]{gans2003telephone,borst2004dimensioning,bassamboo2006design,gurvich2008service, mandelbaum2009staffing,zhang2012staffing,dai2012many,gurvich2013excursion}. Although these analysis provide accurate results for the Erlang A and its generalized models, the facilities represented by these models commonly exhibit a fundamental operational problem: high sensitivity of the QoS properties for systems with customer abandonment. The study conducted by \cite{whitt2006sensitivity} analyzed the sensitivity of performance indicators to changes in the model parameters, and concluded that ``performance is quite sensitive to small percentage changes in the arrival rate or the service rate.'' The fundamental problem is in the square-root staffing rule: the control range of the appropriate number of servers (staff members) to achieve the QED regime is of the order of $\sqrt{R}$, where $R=\lambda/\mu$. Thus, when systems are large (with large $R$), the control range (normalized by the system size) becomes very small. In fact, the QED regime that exists in between two extreme (QD and ED) regimes is rather delicate, as pointed out by \cite{mandelbaum2005palm}. Thus, when arrival and service rates fluctuate (as they often do in practice), the performance of the system could drastically change due to a slight change of $R$.}

\textcolor{black}{In practice, to prevent such fluctuations, operators often try to control the system by changing its characteristics over time. According to \cite{crabill1977classified}, four different control schemes can be considered as effective measures: (1) Control of the number of staff, (2) Control of the arrival rate, (3) Control of the service rate, and (4) Control of the queue discipline. Out of these four possible schemes, (1) and (4) have been analyzed in \cite{feldman2008staffing}, but it is not always possible to quickly change staff or discipline especially when changes in parameters are unexpected. An alternative, more manageable approach could be (2) or (3). As an example of (2), \cite{koccauga2010admission} consider the trade-off between blocking arrivals and server idleness, leading to an optimal threshold queue admission policy. Another example is \cite{sanders2017optimal}: They study an admission control policy within a revenue maximization framework for large-scale systems that operate in the QED regime. Options (2) and (3) are often studied jointly; such examples include \cite{ghosh2010optimal} and \cite{koccauga2017approximating}, whose objective is to minimize long-run average costs. These examples demonstrate the effectiveness of (2) and (3); however, most previous literature focuses on the long-run average cost, and does not attempt to explain how the measures of (2) and (3) could reduce the sensitivity of large-scale abandonment systems. In this paper, we consider (2) and (3)--controlling arrival and/or service rates under congestion--and call these measures the \emph{congestion-based control} (CBC) scheme. Our analysis reveals the root cause of the highly sensitive property of abandonment systems, enabling us to utilize the CBC scheme to make such systems more robust.}

\textcolor{black}{We study a modified Erlang~A model with the CBC scheme, which considers two types of abandonment (reneging/balking). Our analysis employs the Markov chain decomposition approach: Following \cite{sasanuma2019markov}, we decompose the entire system into two sub-systems, an M/M/s/s queue and the reneging/balking queue, and analyze each sub-system separately. The same approach has been utilized to solve other variations of abandonment systems \cite{SASANUMA2021212, sasanuma2021asymptotic}. This decomposition approach has three major analytical benefits. First, it reveals the relationship between the full system and the sub-systems to help us understand how each sub-system contributes to the system performance. Specifically, for this modified Erlang A model, when the two resource requirements ($R$ for the M/M/s/s queue and $R_Q$ for the reneging/balking queue) match, we observe a singularity point in which three regimes (QD, ED, and QED) co-exist in a phase diagram of asymptotic service regimes. Since this singularity is the root cause of the high sensitivity of the full system, we can make large systems more robust by simply making the two resource requirements of sub-systems different and eliminating the singularity; this simple idea is the primary insight into the CBC scheme. The second benefit is that the decomposition method makes the calculation procedure efficient by utilizing the previously known results for the M/M/s/s and the reneging/balking queues; for example, we can utilize the analytical results of sub-chains shown in \cite{SASANUMA2021212} and \cite{sasanuma2021asymptotic}. Lastly, the decomposition approach makes the analysis easier since each decomposed sub-chain is elementary and its approximate analytical properties are simply described by Poisson/normal probability functions.}

\textcolor{black}{In summary, the contributions of this paper are: 1) we derive non-asymptotic and asymptotic representations for the QoS performance indicators of the modified Erlang~A system following the Markov chain decomposition approach; 2) we reveal how the CBC scheme can avoid singularity in the phase diagram of service regimes and make large systems more robust to changes in model parameters; and 3) we show that the CBC scheme converts the square-root staffing rule into the linear staffing rule in the asymptotic limit of large systems.}


\section{Modified Erlang A Model}\label{sec:model}
\subsection{Setup of the Model}
Our modified Erlang A model is a simple extension of the original Elang A model. Figure \ref{fig:MarkovChain} shows the Markov chain (MC) structure of the modified Erlang A reneging model.
\begin{figure}[H]
\centering
{\includegraphics*[scale=0.2]{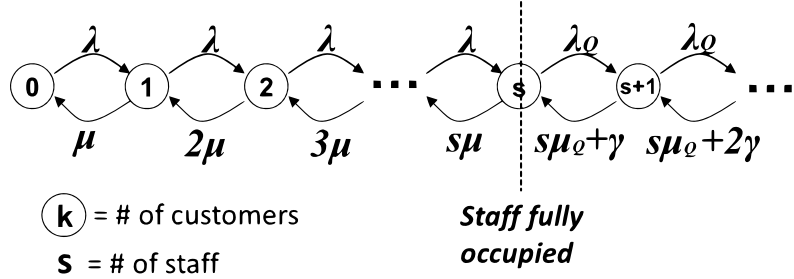}}
\caption{The modified Erlang A model with exponential reneging.}\label{fig:MarkovChain}
\end{figure}
We assume that there are $s$ staff members in the system, the arrival process is Poisson with rate $\lambda$, and the service time is exponential with rate $\mu$. Unlike the original Erlang A reneging model, the modified Erlang A model allows a step change in the baseline arrival and service rates when a system is busy: The arrival rate can drop by a proportion $\varepsilon\:(\geq 0)$, and the service rate can increase/decrease by a proportion $\tau$, where $\varepsilon$ and $\tau$ are altered by a congestion-based control (CBC) policy. The arrival and service rates when a system is busy are defined as $\lambda_{Q} \doteq \left(1-\varepsilon \right)\lambda $ and $\mu _{Q} \doteq \left(1+\tau \right)\mu $, respectively. Additionally, the modified Erlang A model allows customer abandonment either through exponential reneging or state-dependent balking, but not both in the same model. Specifically, for the modified Erlang A reneging model, we assume that each customer in queue reneges after an exponentially distributed time with rate $\gamma\:(>0)$ (note: the $\varepsilon =\tau =0$ case corresponds to the original Erlang A model); for the modified Erlang A balking model, we assume that the arrival rate drops by a linear balking rate $\delta\:(>0)$ for each additional customer in queue. In sum, the birth and death coefficients for the modified Erlang A reneging and balking models are presented as follows. (Note that $(\cdot )^+$ denotes a positive part.)

\begin{enumerate}
\item For the reneging system, the total arrival rate and the total service rate at state $k$ are
\footnotesize
$\lambda _{k} =\left\{\begin{array}{ll} {\lambda } & {0 \le k<s} \\ {\lambda _{Q} } & {s\le k} \end{array}\right. $ and  $\mu _{k} =\left\{\begin{array}{ll} {k\mu } & {1 \le k\le s} \\ {s\mu _{Q} +\left(k-s\right)\gamma } & {s<k} \end{array}\right.$.\\
\normalsize
\item For the balking system, the total arrival rate and the total service rate at state $k$ are
\footnotesize
$\lambda _{k} =\left\{\begin{array}{ll} {\lambda } & {0 \le k<s} \\ {(\lambda _{Q} -\delta \cdot (k-s))^+} & {s\le k} \end{array}\right. $ and  $\mu _{k} =\left\{\begin{array}{ll} {k\mu } & {1 \le k\le s} \\ {s\mu _{Q} } & {s<k} \end{array}\right.$.
\normalsize
\end{enumerate}


\textcolor{black}{We define two resource requirements for the system: $R \doteq \lambda/\mu$ when a system is not busy and $R_Q \doteq \lambda_Q/\mu_Q$ when a system is busy. Since we implement the CBC policy to improve the performance of the system when it is busy, we only consider the case where both $R_Q \leq R$ and $\lambda_Q \leq \lambda$ hold; thus proportions $\varepsilon$ and $\tau$ satisfy $0 \leq \varepsilon \leq 1$, $-1 \leq \tau \leq 1$, and $\varepsilon+\tau \geq 0$. (Note that we can use the same technique to discuss the cases $R_Q > R$ (a system slows down when it is busy) and/or $\lambda_Q > \lambda$ (more customers are attracted to join when a system is busy), but we do not discuss these cases in this paper.) Finally, we assume an independence among parameters: $\gamma$, $\delta$, $\varepsilon$, and $\tau$ do not depend on $\lambda$, $\mu$, $s$, or state $k$.}

\subsection{Decomposition of the Markov Chain}
\textcolor{black}{To solve the modified Erlang A model, we split the entire system into two sub-systems and analyze them separately since the two sub-systems, one when it is busy and the other when it is not, posses very different queueing properties, which are easy to analyze independently but complicated to study jointly.} Specifically, we divide the full MC that represents the modified Erlang A model into an M/M/s/s sub-chain and a reneging/balking sub-chain, which overlap at state $s$. We denote the left sub-chain (M/M/s/s sub-system comprised of states 1 to $s$) as sub-chain 1, and the right sub-chain (reneging/balking sub-system comprised of states $s$ or larger) as sub-chain 2. Customers are put in queue if they enter when the system is in sub-chain 2. We denote the probabilities of being in sub-chains 1 and 2 as $P_{1}$ and $P_{2}$, respectively. We denote the queueing probability (delay probability) as $P_Q\:(=P_{2}$); the abandonment probability (abandonment includes both reneging and balking) as $P_{ab}$; the expected number of customers in queue as $L_{Q}$; and the expected time in queue as $W_Q$. All of these performance indicators can be represented using the steady-state probabilities of state $s$ in the sub-chains and full MC, which are denoted as $\pi _{s}^{1} $, $\pi _{s}^{2} $, and $\pi_s$, respectively. For simplicity, we call $\pi _{s}^{1} $, $\pi _{s}^{2} $, and $\pi _{s}$ the \emph{blocking} probabilities of sub-chains 1, 2, and the full MC, respectively.

For analytical convenience, we introduce three parameters: (1) $p \doteq 1-s\mu_Q/\lambda$, (2) $a = a_{s;R} \doteq (s-R)/R$, and (3) $c=c_{s;R}  \doteq (s-R)/\sqrt R$. First, $p$ represents a heavy-traffic limit of the abandonment probability\footnote{Regardless of how a customer abandons a system, in a heavy-traffic limit, the system accommodates $s\mu_Q$ customers per unit time on average. Hence, the number of customers abandoning a system is $\lambda-s\mu_Q$, from which we obtain the heavy-traffic limit of $P_{ab}$ as $(\lambda-s\mu_Q)/\lambda=1-\mu_Q/\lambda$.} when $p>0$ ($\lambda>s\mu_Q$), but $p$ can also take a negative value depending on parameters. Second, $a$ is a linear (staffing) coefficient of the number of staff $s$ measured relative to $R$, in units of $R$. Finally, $c$ is a square-root (staffing) coefficient of $s$ measured relative to $R$, in units of $\sqrt R$. Note that from the non-negativity requirement on the number of staff $s$, $p \leq1$, $a \geq-1$, and $c \geq -\sqrt R$ must be satisfied.

\subsection{Quality-of-Service Performance Indicators}
For the modified Erlang A model, the steady-state probability $\pi_s$ as well as the QoS performance indicators such as the delay probability $P_Q$, the abandonment probability $P_{ab}$, and the average queue length $L_Q$ can all be represented by the blocking probabilities of the sub-chains, $\pi _{s}^{1} $ and $\pi _{s}^{2} $, as summarized in Lemma~\ref{str-rep}. We call these representations {\it structural} since they reveal how performance indicators are constructed from sub-chains. Note that we obtain the performance indicators for the original Erlang A model  as a special case by setting $\varepsilon=\tau=0$ and $\theta=\gamma$ in Lemma~\ref{str-rep}. All proofs are in the Appendix.

\begin{lemma}
\label{str-rep}
Performance indicators of the modified Erlang A model are represented as follows:
\begin{align}
\label{pi_s}
\frac{1}{\pi _s} &=\frac{1}{\pi _s^1}+\frac{1}{\pi _s^2}-1,\\
\label{P_Q}
{P_Q} &=  \frac{\pi_s}{\pi _s^2}= \frac{\frac{1}{\pi _s^2}}{\frac{1}{\pi _s^1}+\frac{1}{\pi _s^2}-1},\\
\label{P_AB}
{P_{ab}} &= \frac{1+ p\cdot (\frac{1}{\pi _s^2}-1)}{\frac{1}{\pi _s^2}}{P_Q}\\
\label{P_AB2}
&= \frac{1+ p\cdot (\frac{1}{\pi _s^2}-1)}{\frac{1}{\pi _s^1}+\frac{1}{\pi _s^2}-1},\\
\label{L_Q}
L_{Q}&=\frac{\lambda}{\theta} \cdot (P_{ab}-\varepsilon P_Q)\\
\label{L_Q2}
&=\frac{\lambda}{\theta} \cdot \frac{(1-\varepsilon)+ (p-\varepsilon)(\frac{1}{\pi _s^2}-1)}{\frac{1}{\pi _s^1}+\frac{1}{\pi _s^2}-1},
\end{align}
where
\begin{equation}
\label{p}
p = 1-\frac{s\mu_Q}{\lambda} = 1-(1+\tau)(a+1)=-(1+\tau)a-\tau
\end{equation}
and $\theta=\gamma$ (or $\delta$) for a reneging (or balking) system.
\end{lemma}

Among the performance indicators presented in Lemma \ref{str-rep}, we are most interested in the delay probability $P_Q$ and the abandonment probability $P_{ab}$ because a reduction of $P_Q$ yields higher quality of a service system, and a reduction of $P_{ab}$ leads to higher system throughput (i.e., the average number of customers being serviced per unit time: $X=\lambda \cdot (1-P_{ab})$) and thus higher efficiency of the system. However, it may not be possible to achieve a reduction of $P_Q$ and $P_{ab}$ at the same time. To determine if a simultaneous reduction is possible, we can use the following exact relationship, which holds for any modified Erlang A model including the original Erlang A model. In this corollary, we denote $\pi_s^1$ as $P_{block}$.

\begin{corollary}
\label{PQ-Pab}
$P_{Q}$ and $P_{ab}$ for the modified Erlang A model satisfy the following equation:
\begin{equation}
\label{PQ-Pab-Pblock}
P_{ab} =\dfrac{\left(p-P_{block} \right)P_Q +\left(1-p \right)P_{block}}{1-P_{block} }.
\end{equation}
\end{corollary}

From Corollary \ref{PQ-Pab}, we can determine how $P_Q$ and $P_{ab}$ are dependent on abandonment-related parameters, such as $\gamma$ (or $\delta$) and $\varepsilon$:

\begin{corollary}
\label{trade-off}
$P_Q$ and $P_{ab}$ for the modified Erlang A model satisfy the following properties:
\begin{enumerate}
\item $P_Q$ is a monotonically decreasing function of abandonment-related parameters.
\item $P_{ab}$ is a monotonically increasing (decreasing) function of abandonment-related parameters if $p<P_{block}$ ($p>P_{block}$, respectively). $P_{ab}$ is independent of all abandonment-related parameters if $p=P_{block}$.
\end{enumerate}
\end{corollary}

Corollary \ref{trade-off} says that if $p<P_{block}$, a trade-off exists and we need to control the level of abandonment in order to achieve an optimal balance between $P_Q$ and $P_{ab}$. In contrast, if $p>P_{block}$, both $P_Q$ and $P_{ab}$ can be minimized simultaneously by controlling the level of abandonment. The best practice in such a case would be to block all arrivals when all staff members are busy. This case could occur when customers in queue significantly slow down the service speed of the system (or in other words, when waiting customers bring large negative externalities to the system). A similar phenomenon can be observed in a congested traffic network \cite{braess2005paradox}.

We can derive an alternative expression (corollary) to Lemma \ref{str-rep} using the stationary probability $P_{Q-}$ of the system having customers in queue (i.e., the probability that the total number of customers in the system is greater than $s$):
$$P_{Q-} \doteq P_Q-\pi_s=\frac{\frac{1}{\pi _s^2}-1}{\frac{1}{\pi _s^1}+\frac{1}{\pi _s^2}-1}.$$

\begin{corollary}
\label{str-rep-P_Q-}
Performance indicators for the modified Erlang A model are also represented as follows:
\begin{align}
\label{P_Q-P_Q-}
P_Q &= \pi_s+P_{Q-},\\
\label{P_ab-P_Q-}
P_{ab} &=\pi_s+p \cdot P_{Q-},\\
\label{L_Q-P_Q-}
L_Q &= \frac{\lambda}{\theta} \cdot \left((1-\varepsilon)\pi_s+(p-\varepsilon) P_{Q-}\right).
\end{align}
\end{corollary}

\begin{remark}
Both Corollary \ref{str-rep-P_Q-} and Lemma \ref{str-rep} are exact and general. From Corollary \ref{str-rep-P_Q-} we can derive various formulae and approximations, such as performance indicators for the Erlang B and C models, and heavy-traffic approximation. See the Appendix for more discussion.
\end{remark}

\section{Non-asymptotic Poisson-Normal Approximation}\label{sec:perform}
\textcolor{black}{In this section we derive a convenient non-asymptotic representation for performance indicators. We start by expressing all blocking probabilities of sub-chains using a Poisson representation; these probabilities are then converted into a normal representation. Finally, we aggregate blocking probabilities to obtain a non-asymptotic normal representation for the performance indicators of the full system.}

\subsection{Poisson Representation of Blocking Probabilities}
The left sub-chain (an M/M/s/s queue) is a truncated M/M/$\infty$ queue, whose steady-state probabilities are proportional to a Poisson distribution. Using this property, the blocking probability of the left sub-chain, known as the Erlang  B (or Erlang Loss) formula, can be represented as a function of a Poisson CDF and PMF with rate parameter $R$ and index $s$. This index $s$ is implicitly assumed to be an integer because it corresponds to the staffing level. Similarly, the right sub-chain can also be regarded as a truncated M/M/$\infty$ queue if we rescale the rate parameter and index, where again the rescaled index is assumed to be an integer. 
Under this integer constraint, the blocking probability of the right sub-chain can also be represented as a function of a Poisson CDF and PMF. Note that the integer constraints for the staffing level (left sub-chain) and the rescaled staffing level (right sub-chain) are both dropped when we convert the Poisson representation to a normal representation, while the continuity correction terms need to be added to the normal representation in order to account for the inevitable error associated with this discrete-to-continuous Poisson to normal conversion.

\begin{table}
\begin{center}
\caption{PMF/PDF/CDF of R.V.'s}\label{tab:cdf}
\footnotesize
\begin{tabular}{ cccc }\noalign{\smallskip}
\hline \noalign{\smallskip}
& Poisson R.V.  &  Normal R.V. with mean $R$  &\  Standard Normal R.V. \\ \noalign{\smallskip}
& with mean $R$ &   and standard deviation $\sqrt{R}$ & \\ \hline \noalign{\smallskip}
R.V. & $X_P \sim Pois(R)$ & $X_N \sim N(R,R)$ & $Z \sim N(0,1)$ \\ \noalign{\smallskip}
CDF & $F_P(\cdot;R)$ & $F_N(\cdot;R,R)$ & $\Phi(\cdot)$ \\ \noalign{\smallskip}
PMF/PDF & $f_P(\cdot;R)$ & $f_N(\cdot;R,R)$ & $\phi(\cdot)$ \\ \noalign{\smallskip} \hline
\end{tabular}
\end{center}
\end{table}

To prepare for the derivation, we define new R.V.'s and functions in Table 1: Poisson, normal, and standard normal R.V.'s and associated probability mass function (PMF), probability density function (PDF), and cumulative distribution function (CDF). We introduce three Poisson R.V.'s: $X_P$ for the left sub-chain, $X'_P$ and $X''_P$ for the reneging and balking models' right sub-chains, respectively. The distribution of each Poisson R.V. is represented by its rate parameter and staffing level. Using the notation in Table \ref{tab:setup-poisson} and assuming non-negative integral staffing levels, we are able to obtain a Poisson representation for all blocking probabilities:

\begin{table}
\begin{center}
\caption{Setup for Poisson representation}\label{tab:setup-poisson}
\scriptsize
\begin{tabular}{ cccc }
\noalign{\smallskip} \hline \noalign{\smallskip}
Sub-chain & Poisson R.V.  & Rate Parameter  & Staffing level \\ \hline \noalign{\smallskip}
M/M/s/s & $X_P \sim Pois(R)$  & $R \doteq \dfrac{\lambda}{\mu}$  & $s$ \\ \noalign{\smallskip}
Reneging & $X'_P \sim Pois(R')$ &  $R' \doteq \dfrac{\lambda_Q}{\gamma} = (1-\varepsilon)(\mu/\gamma)R $ & $s' \doteq \dfrac{s\mu_Q}{\gamma}=\dfrac{\mu_Q}{\gamma} \cdot s $\\ \noalign{\smallskip}
Balking & $X''_P \sim Pois(R'')$ & $R'' \doteq \dfrac{s\mu_Q}{\delta}=\dfrac{(a+1)\mu_Q}{\delta} \cdot R $ & $s'' \doteq \dfrac{\lambda_Q}{\delta}=\dfrac{1-\varepsilon}{1+\tau} \cdot \dfrac{\mu_Q}{(a+1)\delta} \cdot s $ \\ \noalign{\smallskip} \hline \noalign{\smallskip}
\multicolumn {4} {l} {\emph{Note}: $\lambda _{Q} \doteq (1-\varepsilon )\lambda $ and $\mu _{Q} \doteq (1+\tau )\mu $. From $a \doteq (s-R)/R$, $s=(a+1)R$ holds.}
\end{tabular}\\
\end{center}
\end{table}

\begin{lemma}
\label{block-poisson}
For any (possibly rescaled) positive rate parameters ($R$, $R'$, $R''$) and non-negative integral staffing levels ($s$, $s'$, $s''$), inverse blocking probabilities for the left sub-chain (sub-chain 1) and the right sub-chain (sub-chain 2) are exactly represented as follows:
\begin{enumerate}

\item M/M/s/s sub-chain (left sub-chain): 
\begin{equation}
\frac{1}{\pi _{s}^{1} } =\frac{F_{P}(s;R)}{f_{P}(s;R)},
\end{equation}

\item Reneging sub-chain (right sub-chain):
\begin{equation}
\frac{1}{\pi _{s}^{2} } = 1+\frac{1-F_P(s':R')}{f_P(s':R')},
\end{equation}

\item Balking sub-chain (alternate right sub-chain):
\begin{equation}
\frac{1}{\pi _{s}^{2} } = \frac{F_{P}(s'';R'')}{f_{P}(s'';R'')}.
\end{equation}

\end{enumerate}
\end{lemma}

To convert these Poisson representations to normal, we need to find the relationship between the Poisson CDF/PMF and the normal CDF/PDF, which is discussed next.

\subsection{Normal Representation of Blocking Probabilities}
A Poisson distribution is well approximated by the normal distribution with the same mean and standard deviation as long as the mean (rate parameter) of the Poisson is sufficiently large. This approximation is based on the Central Limit Theorem. Using this property, we can convert the Poisson CDF/PMF to the normal CDF/PDF, and ultimately to the standard normal CDF/PDF.

For convenience, we call ${f_P(s;R)}/(1-F_P(s;R))$ the ``modified'' hazard function of the Poisson distribution\footnote{We use the term ``modified" because the hazard function for a discrete R.V. is usually defined slightly differently as $\Pr(X_P=s)/(1-\Pr(X_P \leq {s-1}))$. See \cite{pinedo2012scheduling}.} and $h(x) \doteq \phi(x)/(1-\Phi(x)) (=\phi(-x)/\Phi(-x))$ the hazard function of the standard normal distribution. We represent the conversion formulae in two different ways: using $R$ and $c$ and also using $a$ and $c$. The former is suitable when making a non-asymptotic analysis (a finite $R$ case) and the latter is suitable when making an asymptotic analysis (an infinite $R$ case).

\begin{proposition}
\label{Poisson-Normal}
For a sufficiently large $R$ and a non-negative integer $s$, the Poisson distribution with mean (rate parameter) $R$ and index $s$ and the ``modified'' hazard function of the Poisson distribution are well-approximated by the standard normal:
\begin{equation}
{F_P}(s;R) \approx \Phi (c_{s;R} + \Delta_R),
\end{equation}
\begin{equation}
{f_P}(s;R) \approx \frac{\phi (c_{s;R} + \Delta_R)}{\sqrt R} = \frac{a_{s;R}\cdot \phi (c_{s;R} + \Delta_R)}{c_{s;R}},
\end{equation}
\begin{equation}  
\frac{f_P(s;R)}{1-F_P(s;R)} \approx \frac{h(c_{s;R} + \Delta_R)}{\sqrt R} = \frac{a_{s;R}\cdot h(c_{s;R} + \Delta_R)}{c_{s;R}},
\end{equation}
and
\begin{equation}
\frac{f_P(s;R)}{F_P(s;R)} \approx \frac{h(-c_{s;R} - \Delta_R)}{\sqrt R} = \frac{a_{s;R}\cdot h(-c_{s;R} - \Delta_R)}{c_{s;R}}.
\end{equation}
\end{proposition}

Observe in Proposition \ref{Poisson-Normal} that while the Poisson representation requires $s$ to be a non-negative integer, the normal representation does not. However, to account for the errors caused by the Poisson-to-normal (discrete-to-continuous) conversion, a continuity correction term $\Delta_R$ appears in the normal representation. This correction term $\Delta_R$ is non-negligible when $R$ is small (for example, if $R$ is around 10), but diminishes to zero if $R$ is large. Note that Proposition \ref{Poisson-Normal} without the $\Delta_R$ term is equivalent to the Poisson-to-normal conversion formulae seen in several textbooks (for example, see \cite{tijms2003first} and \cite{harchol2013performance}). However, we are interested in both small and large $R$, so we maintain $\Delta_R$ in our formulae.

\begin{remark}
Proposition \ref{Poisson-Normal} relies on the Central Limit Theorem that is applied to the sum, $R$, of i.i.d. Poisson R.V.'s with mean 1. Thus, if $R$ is very small, the approximation becomes imprecise. If we need more accurate results for a smaller $R$ case, we can use the Wilson-Hilferty approximation \cite{Wilson01121931,lesch2009some}. According to this approximation, we can convert the Poisson CDF and PMF as follows:
$F_P(s;R) \approx \Phi(z(s,R)),$
$f_P(s;R) \approx \Phi(z(s,R))-\Phi(z(s-1,R)),$
where $z(s,R) \doteq \frac{\sqrt[3]{\frac{R}{s+1}}-1+\frac{1}{9(s+1)}}{\frac{1}{3\sqrt{s+1}}}.$
While this approximation does not support an asymptotic analysis, it is very accurate even for $R=1$ and easy to calculate using a spreadsheet. Hence, the non-asymptotic formulae based on the Wilson-Hilferty approximation provide an important alternative to the non-asymptotic formulae that are based on the Central Limit Theorem for systems with very small means.
\end{remark}

We are now ready to derive the blocking probabilities in a standard normal representation. By combining Lemma \ref{block-poisson} and Proposition \ref{Poisson-Normal} with the notation defined in Table \ref{tab:setup-normal}, we obtain the following lemma. (Note that in this lemma, all integer constraints have been dropped.)

\begin{table}
\begin{center}
\caption{Notation for standard normal representation}\label{tab:setup-normal}
\scriptsize
\begin{tabular}{ ccccc }
\noalign{\smallskip} \hline \noalign{\smallskip}
Sub-chain  & Standard Normal R.V.  & Square-root Coef & Linear Coef & Continuity Correction \\ \hline \noalign{\smallskip}
M/M/s/s & $Z \sim N(0,1)$  & $c \doteq c_{s;R}$  & $a  \doteq a_{s;R}$ & $\Delta  \doteq \Delta_R$ \\ \noalign{\smallskip}
Reneging & $Z' \sim N(0,1)$ & $c'  \doteq c_{s';R'}$ & $a'  \doteq a_{s';R'}$ & $\Delta'  \doteq \Delta_{R'}$\\ \noalign{\smallskip}
Balking & $Z'' \sim N(0,1)$ & $c''  \doteq c_{s'';R''}$ & $a''  \doteq a_{s'';R''}$ & $\Delta''  \doteq \Delta_{R''}$ \\ \noalign{\smallskip} \hline \noalign{\smallskip}
\multicolumn {5} {l} {\emph{Note}: $a_{s;R}  \doteq (s - R)/R$, $c_{s;R}  \doteq (s - R)/\sqrt R$, $\Delta_R  \doteq 0.5/\sqrt R$.}
\end{tabular}\\
\end{center}
\end{table}
 
\begin{lemma}
\label{block-normal}
For sufficiently large (at least around 10) rate parameters ($R$, $R'$, $R''$), inverse blocking probabilities of sub-chains are approximated by the hazard function for the standard normal distribution as follows:

\begin{enumerate}

\item M/M/s/s sub-chain (left sub-chain):  
\begin{equation}
\frac{1}{\pi _{s}^{1} } \approx \frac{\sqrt {R}}{h(-c-\Delta)} \text{ or } \dfrac{1}{a\cdot h(-c-\Delta)/c},
\label{pi_s^1}
\end{equation}  
\item Reneging sub-chain(right sub-chain):
\begin{equation}  
\frac{1}{\pi _{s}^{2} }  \approx 1+\frac{\sqrt {R'}}{h(c'+\Delta ')} \text{ or } 1+\dfrac{1}{a'\cdot h(c'+\Delta ')/c'},
\label{pi_s^2r}
\end{equation}  
\item Balking sub-chain (alternate right sub-chain):
\begin{equation}  
\frac{1}{\pi _{s}^{2} }  \approx \frac{\sqrt {R''}}{h(-c''-\Delta '')} \text{ or } \dfrac{1}{a''\cdot h(-c''-\Delta '')/c''}.
\label{pi_s^2b}
\end{equation}  
\end{enumerate}
\end{lemma}

\subsection{Non-asymptotic Normal Representation of Performance Indicators}
We derive the non-asymptotic normal representation and corresponding staffing rule for the modified Erlang A model. For this purpose, we first find the relationships among key parameters of sub-chains: square-root coefficients, linear coefficients, and rate parameters. To simplify expressions, we introduce a parameter $a_Q \doteq \frac{s-R_Q}{R}=a+\frac{R-R_Q}{R}=a+\frac{\varepsilon+\tau}{1+\tau}$. Using Tables~\ref{tab:setup-poisson} and~\ref{tab:setup-normal}, we obtain Table~\ref{tab:normal-parameter}. Combining Lemma~\ref{str-rep} and Lemma~\ref{block-normal} with the help of Table~\ref{tab:normal-parameter}, we can derive the non-asymptotic normal representation of $\pi_s$ and $P_{Q-}$ as shown in Table \ref{tab:norm-rep}. $P_Q$ and $P_{ab}$ are obtained from these $\pi_s$ and $P_{Q-}$ using Corollary~\ref{str-rep-P_Q-}. Note that these non-asymptotic formulae cannot be represented by a single parameter $c$ or $a$ in contrast to the case for the popular Erlang A square-root staffing rule since the non-asymptotic formulae depend on the size of the system (thus, requiring any two of the three parameters: $a$, $c$, and $R\:(=(c/a)^2)$).
\begin{table}
\begin{center}
\caption{Parameters characterizing normal representation}\label{tab:normal-parameter}
\scriptsize
\begin{tabular}{ cccc }
\hline \noalign{\smallskip}
Sub-chain & Square-root Coefficient & Linear Coefficient & Rate Parameter\\ \hline \noalign{\smallskip}
M/M/s/s & $c\:(=a \sqrt R)$  & $a\:(=c/\sqrt R)$ & $R\:\left(=(c/a)^2\right)$\\ \noalign{\smallskip}
Reneging & $c'=\dfrac{a_Q}{a} \sqrt{\dfrac{1+\tau}{1-\varepsilon}} \sqrt{\dfrac{\mu_Q}{\gamma}}\cdot c$  & $a'=\dfrac{1+\tau}{1-\varepsilon} \cdot a_Q$ & $R'=\dfrac{1-\varepsilon}{1+\tau} \cdot \dfrac{\mu_Q}{\gamma} R=\dfrac{(1-\varepsilon)\mu}{\gamma}R$\\ \noalign{\smallskip}
Balking & $c''=-\dfrac{a_Q}{a}\sqrt{\dfrac{\mu_Q}{(a+1)\delta}}\cdot c$  & $a''=- \dfrac{a_Q}{a+1} $ & $R''=\dfrac{(a+1)\mu_Q}{\delta}R$\\ \hline \noalign{\smallskip}
\multicolumn {4} {l} {\emph{Note}: $a_Q=a$ and $(1+\tau)/(1-\varepsilon)=1$ hold if $\varepsilon+\tau=0$ (i.e., $R=R_Q$).}
\end{tabular}\\ 
\end{center}
\end{table}

\begin{table}
\begin{center}
\caption{Non-asymptotic normal representation of the Modified Erlang A model}\label{tab:norm-rep}
\scriptsize
\begin{tabular}{ ccc }
\hline \noalign{\smallskip}
Performance &  &   \\
Indicator & Reneging system & Balking system  \\ \hline \noalign{\smallskip}
$\pi_s$ & $\frac{\dfrac{1}{\sqrt R}}{\dfrac{1}{h(-c-\Delta)}+ \dfrac{\sqrt{(1-\varepsilon)\mu/\gamma}}{h(c'+\Delta')}}$  &  $\dfrac{\dfrac{1}{\sqrt R}}{\dfrac{1}{h(-c-\Delta)}+ \dfrac{\sqrt{(a+1)\mu_Q/\delta}}{h(-c''-\Delta'')}-\dfrac{1}{\sqrt R}}$  \\ \noalign{\smallskip} \noalign{\smallskip} \noalign{\smallskip}
$P_{Q-}$ & $\frac{\dfrac{\sqrt{(1-\varepsilon)\mu/\gamma}}{h(c'+\Delta')}}{\dfrac{1}{h(-c-\Delta)}+ \dfrac{\sqrt{(1-\varepsilon)\mu/\gamma}}{h(c'+\Delta')}}$  &  $\dfrac{\dfrac{\sqrt{(a+1)\mu_Q/\delta}}{h(-c''-\Delta'')}-\dfrac{1}{\sqrt R}}{\dfrac{1}{h(-c-\Delta)}+ \dfrac{\sqrt{(a+1)\mu_Q/\delta}}{h(-c''-\Delta'')}-\dfrac{1}{\sqrt R}}$ \\  \noalign{\smallskip} \hline  \noalign{\smallskip}
\end{tabular}
\end{center}
\end{table}

The decision rule for the optimal number of staff that satisfies a specific QoS target is obtained as a direct application of the non-asymptotic normal representation of performance indicators:

\begin{proposition}
\label{sq-root-staff}
Assume that a specific QoS requirement is either $P_{Q} <\alpha \left(<1\right)$ or $P_{ab} <\alpha \left(<1\right)$, where the forms of $P_Q$ and $P_{ab}$ are specified as $P_Q=\pi_s+P_{Q-}$ and $P_{ab}=\pi_s+p \cdot P_{Q-}$ using $\pi_s$ and $P_{Q-}$ in Table \ref{tab:norm-rep}. Then, given $R$, the minimum number of staff $s_{\alpha;R}$ that is required to guarantee the QoS requirement is obtained as $s_{\alpha;R} = \lceil {R+c_{\alpha;R} \sqrt{R}} \rceil$ ($ s_{\alpha;R} = \lceil {R+a_{\alpha;R} R} \rceil$), where $c_{\alpha;R}$ ($a_{\alpha;R}$, respectively) is the unique solution to either $P_Q=\alpha$ or $P_{ab}=\alpha$.
\end{proposition}


\section{Asymptotic Normal Approximation}\label{sec:asymptotic}
In this section we derive analytical expressions for the asymptotic properties of the modified Erlang A model in each regime (i.e., QD, QED, or ED) and obtain phase diagrams of service regimes in the asymptotic limit. We see that under CBC, the modified Erlang A model has a wider QED regime than the original Erlang A model; this wider QED regime is described by a linear staffing rule (parameterized by $a$) as opposed to the square-root staffing rule (parameterized by $c$) of the original Erlang A model.

\subsection{Asymptotic Normal Representation of Performance Indicators}
Before we start the derivation, recall that in Lemma \ref{str-rep}, performance indicators are represented by inverse blocking probabilities of sub-chains. \textcolor{black}{We thus need to find the asymptotic limit of each blocking probability on an appropriate scale. Notice that each sub-chain exhibits its congestion property--either saturated or idling--depending on whether the staffing level (number of servers or staff members) is below or above the sub-chain's scale parameter, respectively. Since each sub-chain is indexed by its own scale parameter ($R$ for the left sub-chain and $R_Q$ for the right sub-chain), under the assumption ${R_Q \le R}$, we can identify that, in the asymptotic limit of large systems, the modified Erlang A model as a whole should exhibit a QD regime when we maintain ${s<R_Q}$ (both sub-systems are mostly idling) or an ED regime when we maintain $R<s$ (both sub-systems are mostly saturated). The most interesting case, leading to the QED regime, is when (1) ${R_Q < s < R}$ (under the condition that ${R_Q<R}$ or equivalently, $\varepsilon+\tau > 0$), in which the left sub-chain is mostly saturated while the right sub-chain is mostly idling, or (2) ${s \approx R=R_Q}$  (under the condition that ${R_Q=R}$ or equivalently, ${\varepsilon+\tau = 0}$), in which both sub-chains are neither idling nor saturated.}

\textcolor{black}{For the first~(${R_Q<s<R}$) case, since there is a wider region in which $s$ achieve the QED regime, we utilize a linear staffing representation and express blocking probabilities using a linear coefficient~$a$; see Table \ref{tab:asymptotic limit of inverse blocking probabilities} for the analytical results of inverse blocking probabilities on a linear scale. For the second~(${s \approx R=R_Q}$) case, to explain the congestion properties of large systems, we utilize a finer, square-root staffing representation, and express blocking probabilities using a square-root coefficient~$c$; see Table \ref{tab:asymptotic limit of inverse blocking probabilities at a phase boundary} for the analytical results of inverse blocking probabilities on a square-root scale.}

By plugging the inverse blocking probabilities in Tables~\ref{tab:asymptotic limit of inverse blocking probabilities} and~\ref{tab:asymptotic limit of inverse blocking probabilities at a phase boundary} into Equations~\eqref{P_Q} and~\eqref{P_AB2} in Lemma~\ref{str-rep}, the asymptotic representation of performance indicators for the modified Erlang A model is derived in Table \ref{tab:asymptotic representation of Erlang A}. Note that~$\phi(c)$ in Table~\ref{tab:asymptotic representation of Erlang A} is defined as
\begin{equation}
\phi(c) \doteq \dfrac{\frac{\sqrt{\mu_Q/\theta}}{h \left(\sqrt{\mu_Q/\theta}\cdot c \right)}}{\frac{1}{h \left(-c\right)}+\frac{\sqrt{\mu_Q/\theta}}{h \left(\sqrt{\mu_Q/\theta}\cdot c \right)}},
\label{phi}
\end{equation}
where $\theta=\gamma$ (or $\delta$) for a reneging (or balking) system.

\textcolor{black}{To summarize this subsection, in the QED regime, we obtain a square-root staffing rule when both sub-chains change their congestion properties at the same threshold~${s \approx R=R_Q}$, while we obtain a linear staffing rule when the left sub-chain is mostly saturated and the right sub-chain is mostly idling when~${R_Q <s< R}$.}

\begin{table}
\begin{center}
\caption{Asymptotic limit of blocking probabilities on a linear scale when $R_Q<R$}\label{tab:asymptotic limit of inverse blocking probabilities}
\scriptsize
\begin{tabular}{ cccc }
\hline\noalign{\smallskip}
sub-chain & $0 \leq s < R_Q$ & $R_Q<s<R$ & $R<s$  \\ \hline\noalign{\smallskip}
M/M/s/s& $a<0, c \to -\infty, \dfrac{1}{\pi_s^1} \to -\dfrac{1}{a}$ & same as the left & $a>0, c \to +\infty, \dfrac{1}{\pi_s^1} \to +\infty$ \\ \noalign{\smallskip}
Reneging& $a'<0, c' \to -\infty, \dfrac{1}{\pi_s^2} \to +\infty$ & same as the right & $a'>0, c' \to +\infty, \dfrac{1}{\pi_s^2} \to \dfrac{a+1}{a_Q}$\\ \noalign{\smallskip}
Balking& $a''>0, c'' \to +\infty, \dfrac{1}{\pi_s^2} \to +\infty$ & same as the right & $a''<0, c'' \to -\infty, \dfrac{1}{\pi_s^2} \to \dfrac{a+1}{a_Q}$\\\noalign{\smallskip}\hline \noalign{\smallskip}
\end{tabular}\\
\end{center}
\end{table}

\begin{table}
\begin{center}
\caption{Asymptotic limit of blocking probabilities on a square-root scale when $R=R_Q$}\label{tab:asymptotic limit of inverse blocking probabilities at a phase boundary}
\scriptsize
\begin{tabular}{ cccc }
\hline\noalign{\smallskip}
sub-chain & $0 \leq s < R_Q$ & $s \approx R=R_Q$ ($s=R+c\sqrt R$) & $R<s$  \\ \hline\noalign{\smallskip}
M/M/s/s& *&$c$, $a \to 0$, $ \dfrac{1}{\pi_s^1} \to +\infty$, $\dfrac{a}{\pi_s^1} \to \dfrac{1}{h(-c)/c}$& * \\ \noalign{\smallskip}
Reneging& *&$c'=\sqrt{\dfrac{\mu_Q}{\gamma}} \cdot c$, $a' \to 0$, $\dfrac{1}{\pi_s^2} \to +\infty$, $\dfrac{a}{\pi_s^2} \to \dfrac{1}{h(c')/c'}$& * \\ \noalign{\smallskip}
Balking& *&$c''\to -\sqrt{\dfrac{\mu_Q}{\delta}} \cdot c$, $a'' \to 0$, $\dfrac{1}{\pi_s^2} \to +\infty$, $\dfrac{a}{\pi_s^2} \to -\dfrac{1}{h(-c'')/c''}$& * \\ \noalign{\smallskip} \hline \noalign{\smallskip}
\multicolumn {4} {l} {\emph{Note}: The symbol * indicates the same convergence properties as Table \ref{tab:asymptotic limit of inverse blocking probabilities} except that $a_Q=a$ holds }\\
\multicolumn {4} {l} {in this table (Table \ref{tab:asymptotic limit of inverse blocking probabilities at a phase boundary}), where $R=R_Q$ (i.e., $\varepsilon+\tau=0$ and thus $a_Q=a$) is assumed.}
\end{tabular}\\
\end{center}
\end{table}

\begin{table}
\begin{center}
\caption{Asymptotic representation of the modified Erlang A}\label{tab:asymptotic representation of Erlang A}
\scriptsize
\begin{tabular}{ ccccccccc }
\hline\noalign{\smallskip}
&&ED Regime&& \multicolumn {3} {c} {QED Regime}&& QD Regime\\ \cline{3-3} \cline{5-7} \cline{9-9} \noalign{\smallskip}
Perf && $0 \leq s < R_Q$ && $R_Q<s<R$ && $s \approx R$ ($s=R+c\sqrt R$)&& $R<s$ \\ \noalign{\smallskip}
Ind &&  && (when $R_Q<R$) && (when $R_Q=R$) &&  \\ \hline \noalign{\smallskip}
$P_Q $&& 1 && $\dfrac{1}{1-\dfrac{a_Q}{a}}=\dfrac{1-\dfrac{s}{R}}{1-\dfrac{R_Q}{R}}$ && $\phi(c)$&&0 \\ \noalign{\smallskip} \noalign{\smallskip}
$P_{ab} $&& $p$ && $\dfrac{\varepsilon}{1-\dfrac{a_Q}{a}}=\dfrac{\varepsilon\left(1-\dfrac{s}{R}\right)}{1-\dfrac{R_Q}{R}}$ && $\varepsilon \phi(c)$ && 0 \\ \noalign{\smallskip} \hline \noalign{\smallskip}
\multicolumn {9} {l} {\emph{Notes}: $p \doteq 1-(1+\tau)(a+1)$, $a_Q \doteq \frac{s-R_Q}{R}=a+\frac{\varepsilon+\tau}{1+\tau}$.}
\end{tabular}\\
\end{center}
\end{table}

\begin{remark}
\textcolor{black}{The process to derive asymptotic formulae shows two important distinctions between small systems and large systems. First, $\pi_s$ in Corollary~\ref{str-rep-P_Q-} (non-asymptotic representation) is dropped in Table~\ref{tab:asymptotic representation of Erlang A} (asymptotic representation) because $\pi_s$ in Table~\ref{tab:norm-rep} approaches zero as $R$ increases. Second, all continuity correction terms in Table~\ref{tab:norm-rep} become negligible as $R$ increases. These observations imply that inaccuracies for asymptotic formulae are due to the discrete state space of Markov chains. For small systems, the discreteness of the state space of Markov chain impacts the values of performance indicators, while for large systems we can ignore the contribution from any single state since its contribution is negligible when systems are large. We examine the difference between non-asymptotic and asymptotic formulae using numerical experiments in Section~\ref{sec:numerical}.}
\end{remark}

\begin{remark}
Although the modified Erlang A reneging and balking models have different non-asymptotic representations as shown in Table~\ref{tab:norm-rep}, their asymptotic representations exactly match in Table~\ref{tab:asymptotic representation of Erlang A}. This result implies that both (exponential) reneging and (linear) balking asymptotically contribute to delay and abandonment probabilities in the same way, although they are often discussed in different contexts. This coincidence is intuitively explained as follows: If the size of a system increases, the discreteness a system exhibits becomes smaller, continuity correction terms diminish, $\pi_s$ diminishes, and the time intervals between reneging (or balking) diminish. In this case, given $\gamma=\delta$, there is no difference in the impact of reneging and balking to the QoS properties ($P_Q$ and $P_{ab}$) of the system because customers constantly exit from the system with the same rate proportional to the size of the queue (i.e., a fluid approximation applies). However, note that the delay ($W_Q$) is different between reneging and balking systems; a balking system has a lower $W_Q$ because fewer customers join the system than a reneging system.
\end{remark}

\subsection{Phase Diagram}
\textcolor{black}{To understand the congestion properties of the modified Erlang A model in the asymptotic limit of large systems, it is convenient to observe the phase diagram of the three service regimes. In this phase diagram we do not distinguish between reneging and balking models since they share the same asymptotic results. We use an index~${1-\frac{R_Q}{R}}$ to represent the level of intervention, which ranges from~0 (no intervention case when~${\frac{R_Q}{R}=\frac{1-\varepsilon}{1+\tau}=1}$) to~1 (full intervention case when~${R_Q=1-\varepsilon=0}$). The ``no intervention'' case corresponds to the CBC scheme with~${\varepsilon+\tau=0}$, which includes the original Erlang~A model~(${\varepsilon=\tau=0}$) as a special case. The ``full intervention'' case corresponds to the CBC scheme with $\varepsilon=1$ (customers' arrivals are blocked when all servers are full), which makes the modified Erlang~A model equivalent to the Erlang B model. The phase boundaries of asymptotic service regimes in these diagrams are~${s=R}$ and~${s=R_Q\:(=\frac{1-\varepsilon}{1+\tau}R \leq R)}$ since the modified Erlang~A model contains two sub-systems indexed by resource requirements~$R$ and~$R_Q$. Thus, for a given level of intervention~${1-\frac{R_Q}{R}}$, a system achieves QED in the asymptotic limit when the staffing level (measured in $R$) is ${\frac{R_Q}{R}\:(=\frac{1-\varepsilon}{1+\tau})<\frac{s}{R}<\frac{R}{R}\:(=1)}$ (see Figure~\ref{fig:PD-General-StaffingLevel}), or equivalently, when the traffic intensity is ${\frac{R}{R}\:(=1)<\frac{R}{s}<\frac{R}{R_Q}\:(=\frac{1+\tau}{1-\varepsilon})}$ (see Figure~\ref{fig:PD-General-TrafficIntensity}). For simplicity of presentation, we denote the modified Erlang~A with~${\varepsilon+\tau=0}$ (${R=R_Q}$) as the Erlang~Ae in Figures~\ref{fig:PD-General-StaffingLevel} and~\ref{fig:PD-General-TrafficIntensity}. Erlang~A~(${\varepsilon=\tau=0}$) is a special case of Erlang~Ae in this case.}
\begin{figure}[H]
\centering
{\includegraphics*[scale=0.20]{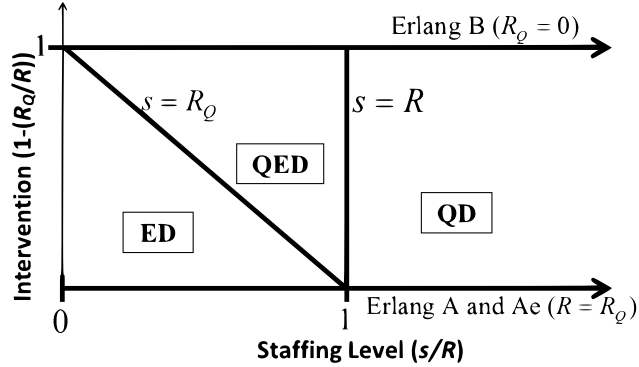}}
\caption{Phase diagram (staffing level representation) for the modified Erlang A model.}\label{fig:PD-General-StaffingLevel}
\end{figure}
\begin{figure}[H]
\centering
{\includegraphics*[scale=0.20]{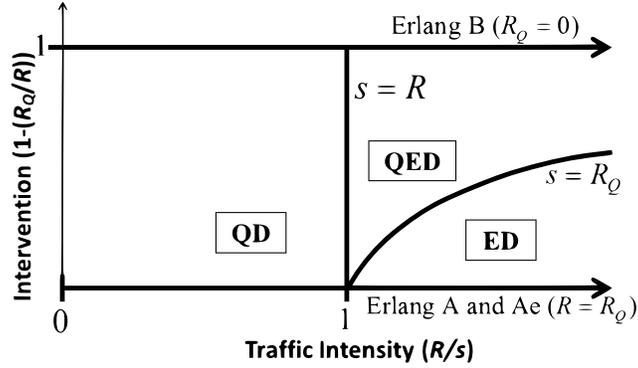}}
\caption{Phase diagram (traffic intensity representation) for the modified Erlang A model.}\label{fig:PD-General-TrafficIntensity}
\end{figure}

\textcolor{black}{These phase diagrams not only reveal asymptotic service regimes according to a given staffing level and intervention, but also provide visual information about robustness (i.e., sensitivities to changes in parameters). Specifically, we observe in Figures~\ref{fig:PD-General-StaffingLevel} and~\ref{fig:PD-General-TrafficIntensity} that $s=R$ when $R=R_Q$ (no intervention case; the bottom edge of the phase diagram) is a singular point, where all three regimes co-exist at a single point. This is the reason why the control of the staffing level~$s$ needs to be delicate (and requires a square-root staffing rule) to achieve QED for the original Erlang A model; as pointed out by \cite{whitt2006sensitivity}, a small change in traffic intensity~${{R}/{s}\:(={\lambda}/{(s\mu)})}$ due to changes in parameters ($\lambda$, $\mu$, and $s$) can incur a large change in the delay probability $P_Q$ (and thus changes the regime the system belongs to). However, such a delicate control may not be necessary since we know that there exists a wider QED region above the no intervention line (${R=R_Q}$) in phase diagrams. In other words, we can expand the QED regime by implementing the CBC scheme to make $R_Q<R$, avoiding the operation of large systems at a singular point $s\approx R=R_Q$. For practitioners struggling to maintain stable operation of large-scale abandonment systems in the QED regime, it would be easier to control $\varepsilon$ and $\tau$ appropriately to realize~${R_Q<R}$ and make systems robust, than to constantly control~$s$ around~$R=R_Q$ following the conventional or refined square-root staffing rules.}

\section{Numerical Experiments}\label{sec:numerical}
\textcolor{black}{In this section we demonstrate the precision of our non-asymptotic normal approximation, identify the source of discrepancies between non-asymptotic and asymptotic approximations, and observe the robustness of the system with customer abandonment under the CBC scheme.}

\subsection{Comparison among Exact Result and Non-Asymptotic/Asymptotic Approximations}
We first discuss the $R=R_Q$ case. Specifically, we consider the original Erlang A (${\varepsilon=\tau=0}$) case since the popular square-root staffing rule is available. \textcolor{black}{(Note: Our asymptotic representation of $P_Q$ for the $\varepsilon=\tau=0$ case is identical to the popular square-root staffing rule for the original Erlang~A model: $P_Q=\phi(c)$, where $\phi(c)$ is defined in Equation~\eqref{phi}. See also Table~\ref{tab:asymptotic representation of Erlang A}.)} Table~\ref{tab:GMR} shows the difference between the exact staffing levels and the staffing levels derived from the non-asymptotic staffing rule (Proposition \ref{sq-root-staff}) and the square-root staffing rule for the original Erlang A model. The staffing levels derived from the non-asymptotic staffing rule usually agree with the exact staffing levels with a difference of at most 1 at all levels of $\gamma$, while the staffing levels derived from the square-root staffing rule disagree with the exact staffing levels by two or more when $\gamma=10$.

\begin{table}[H]
\begin{center}
\caption{\textcolor{black}{Staffing levels to meet the target $P_Q$ following the exact, non-asymptotic, and square-root staffing rules ($R=R_Q$ case).}}\label{tab:GMR}
\scriptsize
\begin{tabular}{ ccccc }
\hline\noalign{\smallskip}
&&& \multicolumn {2} {c} {Difference from Exact Staffing Level} \\
\cline{4-5} \noalign{\smallskip}
$\gamma$ & Target $P_Q$ level ($\alpha$) & Exact Staffing Level ($s$) & Non-Asymptotic & Square-root Staffing\\ \hline \noalign{\smallskip}
10& 95\% & 20 & $-1$ & $-8$ \\ 
    & 83\% & 30 & 0 & $-5$ \\
    & 60\% & 40 & +1 & $-2$ \\
    & 30\% & 50 & 0 & $-2$  \\ \noalign{\smallskip} \hline \noalign{\smallskip}
1& 95\% & 40 & $-1$ & $-1$ \\ 
    & 83\% & 44 & 0 & 0 \\
    & 60\% & 49 & 0 & +1 \\
    & 30\% & 55 & 0 & $-1$  \\ \noalign{\smallskip} \hline \noalign{\smallskip}
0.1& 95\% & 48 & 0 & 0 \\
    & 83\% & 50 & 0 & 0 \\
    & 60\% & 52 & 0 & 0 \\ 
    & 30\% & 56 & 0 & 0  \\ \noalign{\smallskip} \hline \noalign{\smallskip}
\multicolumn {5} {l} {\emph{Notes}: $\lambda=50$, $\mu=1$, ${\varepsilon=\tau=0}$, $R=R_Q=50$}
\end{tabular}\\
\end{center}
\end{table}

Figure~\ref{fig:GMR-a} compares the numerical results of the non-asymptotic and square-root staffing approximations with the exact values using the same parameters shown in Figure 4 of \cite{garnett2002designing}. As \cite{garnett2002designing} show, the square-root staffing approximation matches well with the exact values at $\gamma=0.1$ and $\gamma=1$, but not at $\gamma=10$. In contrast, our non-asymptotic approximation traces the exact values for all cases.
\begin{figure}[H]
\centering
{\includegraphics*[scale=0.45]{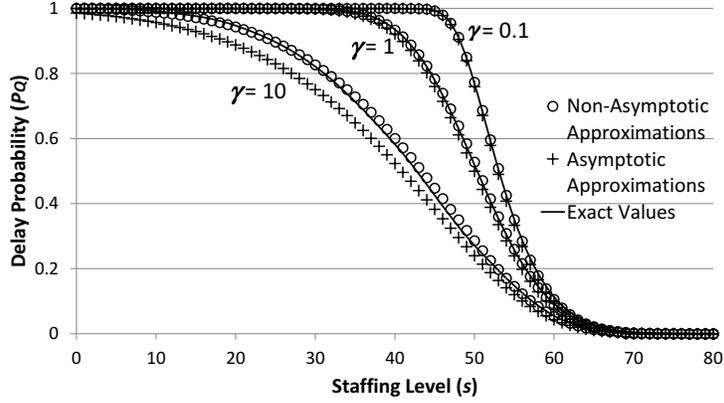}}
\caption{\textcolor{black}{Comparison between the exact values and non-asymptotic/square-root staffing approximations of $P_Q$ ($R=R_Q$ case). \emph{Notes}: $\lambda=50$, $\mu=1$, $\varepsilon=\tau=0$, $R=R_Q=50$.}}\label{fig:GMR-a}
\end{figure}

Figure~\ref{fig:GMR-b} shows the difference of the delay probability ($P_Q$) between our non-asymptotic approximation and the square-root staffing rule. This difference is primarily accounted for by the term $\pi_s$, which only appears in the non-asymptotic representation of $P_Q$ (except for the continuity correction terms $\Delta$ and $\Delta'$, whose contributions are smaller than $\pi_s$). (See the Appendix for the derivation.) The discrepancy term $\pi_s$ could become large when $R$ is small and $\gamma$ is large, in which case $P_Q$ of the square-root staffing rule could have a larger error. On the other hand, in the asymptotic limit of a large $R$, both $\pi_s$ and the continuity correction terms are regarded as negligible (because no discreteness exists), in which case the square-root staffing rule becomes precise.
\begin{figure}[H]
\centering
{\includegraphics*[scale=0.45]{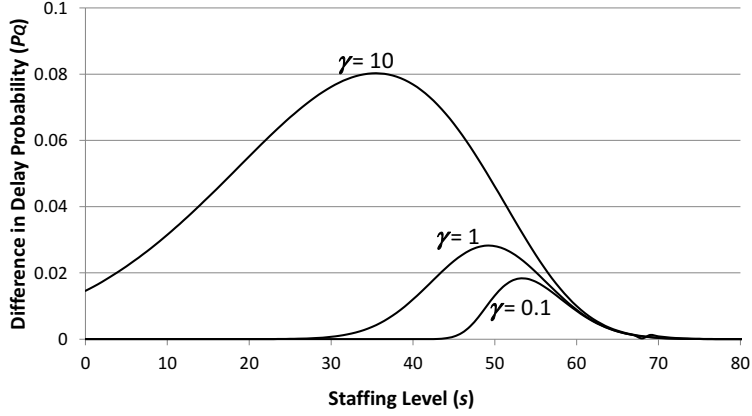}}
\caption{\textcolor{black}{Impact of reneging rate $\gamma$ on the \emph{difference} between non-asymptotic and square-root staffing approximations of $P_Q$  ($R=R_Q$ case). \footnotesize \emph{Notes}: $\lambda=50$, $\mu=1$, $\varepsilon=\tau=0$, $R=R_Q=50$. Ripples at around $s=69$ are due to numerical inaccuracy.}}\label{fig:GMR-b}
\end{figure}

\textcolor{black}{Figure \ref{fig:GMR-c} shows that as $R$ increases, the non-asymptotic approximation approaches the asymptotic approximation, as expected.}
\begin{figure}[H]
\centering
{\includegraphics*[scale=0.4]{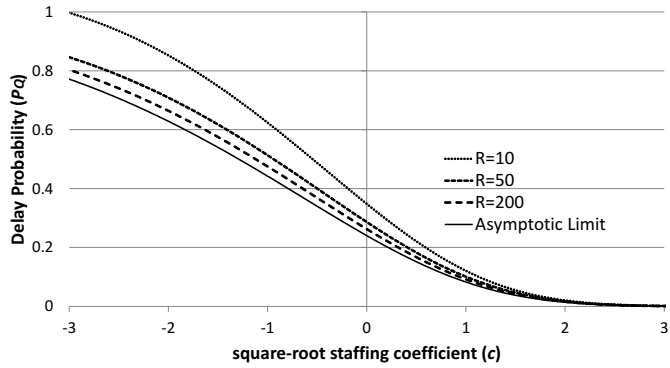}}
\caption{\textcolor{black}{Comparison between non-asymptotic and asymptotic approximations of $P_Q$ (square-root scale; $R=R_Q$ case). \emph{Notes}: $\mu=1$, $\gamma=10$, $\lambda=R=$10, 50, 200, and $\varepsilon=\tau=0$.}}\label{fig:GMR-c}
\end{figure}

\textcolor{black}{Next, we compare the exact values of $P_Q$ with non-asymptotic approximations of $P_Q$ when $R \neq R_Q$ ($\varepsilon+\tau \neq 0$). Table~\ref{tab:GMR-d} shows both absolute error (denoted as Abs.; defined as $|exact - approximation|$) and relative error (denoted as Rel.; defined as $(exact-approximation)/exact$ in \%). Note that we only consider non-asymptotic approximation in this table since the square-root staffing rule is not applicable when $R \neq R_Q$ and the linear staffing approximation is crude to compare with the exact values. Table~\ref{tab:GMR-d} demonstrates that our non-asymptotic approximation is accurate for a wide range of $s$ (from the ED to the QD regime). We observe a larger relative error for a larger $s\:(>R=50)$; however, this is not important since this large relative error is caused by $P_Q$ approaching zero as $s$ increases.}

\begin{table}[H]
\begin{center}
\caption{\textcolor{black}{Comparison between the exact results and non-asymptotic approximations of $P_Q$.}}\label{tab:GMR-d}
\scriptsize
\begin{tabular}{ccccccccccc}
\hline\noalign{\smallskip}
$\varepsilon$, $\tau$ & $s$          & 20       & 30       & 40       & 50      & 60        & 70        & 80        & Average & Max \\ \noalign{\smallskip} \hline \noalign{\smallskip}
0, 0         & Exact      & 1.00    & 1.00    & 0.94    & 0.52   & 0.09     & 0.00     & 0.00     &         &         \\
            & Non-Asym & 1.00     & 1.00    & 0.93    & 0.53   & 0.09     & 0.00     & 0.00     &         &         \\
            & Abs.         & 0.00    & 0.00    & 0.00    & 0.01   & 0.00     & 0.00     & 0.00     & 0.002    & 0.009    \\ 
                        & Rel.\%     & 0.00  & 0.07 & 0.22 & -1.79 & 5.15  & 37.37 & 76.73 &         &         \\
\noalign{\smallskip} \hline \noalign{\smallskip}
0, 0.2       & Exact       & 1.00    & 0.99    & 0.79    & 0.35   & 0.06     & 0.00     & 0.00     &         &         \\
            & Non-Asym  & 1.00    & 0.99    & 0.80    & 0.36   & 0.06     & 0.00     & 0.00     &         &         \\
            & Abs.        & 0.00    & 0.00    & 0.01    & 0.01   & 0.00     & 0.00     & 0.00     & 0.003    & 0.008    \\
                        & Rel.\%     & 0.00  & 0.18 & -0.92  & -2.38 & 5.89  & 37.70 & 76.80 &         &         \\
 \noalign{\smallskip} \hline \noalign{\smallskip}
0.2, 0       & Exact      & 1.00    & 0.97    & 0.73    & 0.32   & 0.06     & 0.00     & 0.00     &         &         \\
            & Non-Asym & 1.00       & 0.97    & 0.74    & 0.33   & 0.05     & 0.00     & 0.00     &         &         \\
            & Abs.         & 0.00    & 0.00    & 0.01    & 0.01   & 0.00     & 0.00     & 0.00     & 0.003    & 0.010    \\
                        & Rel.\%    & 0.02 & 0.19 & -1.34  & -2.56 & 5.86  & 37.66 & 76.79 &         &         \\
 \noalign{\smallskip} \hline \noalign{\smallskip}
0.2, 0.2     & Exact        & 1.00    & 0.91    & 0.59    & 0.24   & 0.05     & 0.00     & 0.00     &         &         \\
            & Non-Asym    & 1.00    & 0.91    & 0.61    & 0.25   & 0.04     & 0.00     & 0.00     &         &         \\
            & Abs.           & 0.00    & 0.00    & 0.01    & 0.01   & 0.00     & 0.00     & 0.00     & 0.003    & 0.012    \\
                        & Rel.\%       & 0.07 & -0.32  & -2.03  & -2.32 & 6.34  & 37.86 & 86.37 &         &         \\
 \noalign{\smallskip} \hline \noalign{\smallskip}
0.2, 0.5     & Exact       & 0.99    & 0.80    & 0.48    & 0.20   & 0.04     & 0.00     & 0.00     &         &         \\
            & Non-Asym    & 0.99    & 0.81    & 0.49    & 0.20   & 0.03     & 0.00     & 0.00     &         &         \\
            & Abs.         & 0.00    & 0.01    & 0.01    & 0.00   & 0.00     & 0.00     & 0.00     & 0.003    & 0.011    \\
                        & Rel.\%       & 0.08 & -1.07  & -2.19  & -2.04 & 8.81  & 38.19 & 76.92 &         &         \\
 \noalign{\smallskip} \hline \noalign{\smallskip}
0.5, 0.2     & Exact       & 0.92    & 0.68    & 0.40    & 0.16   & 0.03     & 0.00     & 0.00     &         &         \\
            & Non-Asym    & 0.93    & 0.69    & 0.41    & 0.17   & 0.02     & 0.00     & 0.00     &         &         \\
            & Abs.       & 0.00    & 0.01    & 0.01    & 0.00   & 0.01     & 0.00     & 0.00     & 0.005    & 0.012   \\
                        & Rel.\%   & -0.45  & -1.60  & -2.37  & -2.02 & 37.96 & 38.19 & 76.92 &         &         \\
 \noalign{\smallskip} \hline \noalign{\smallskip}
            \multicolumn {11} {l} {\emph{Notes}: \textcolor{black}{$\lambda=50$, $\mu=1$, $R=50$. Note that $R \neq R_Q$ except for the first case ($\varepsilon=\tau=0$).}}
\end{tabular}\\
\end{center}
\end{table}

\subsection{Robustness of the Modified Erlang A Model}
\textcolor{black}{In this final subsection, we discuss how the CBC scheme--blocking arrivals ($\varepsilon$) and increasing service rate ($\tau$)--affects the robustness of the modified Erlang~A system. We first observe the impacts of $\varepsilon$ and $\tau$ on $P_Q$ when fixing the size of the system. Figure~\ref{fig:PQ-under-CBC} shows that $\varepsilon$ has larger impact on $P_Q$ than $\tau$, partly due to the higher marginal reduction of $R_Q=\frac{1-\varepsilon}{1+\tau}R$ for $\varepsilon$ than $\tau$. This result is consistent with our intuition that blocking arrivals is more effective than increasing service rates of servers.}

\begin{figure}[H]
\centering
{\includegraphics*[scale=0.45]{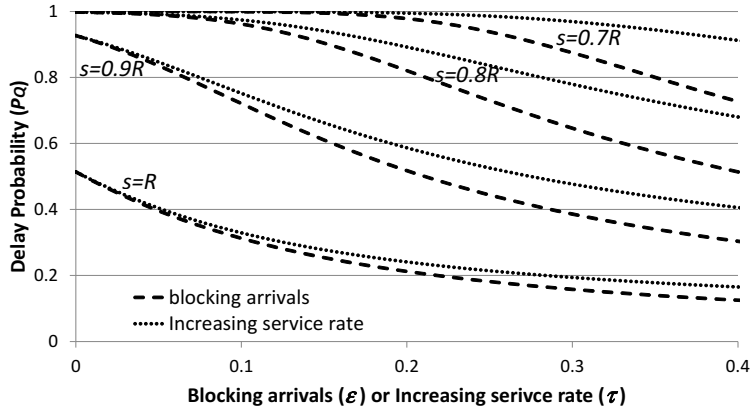}}
\caption{\textcolor{black}{Comparison between the impacts of blocking arrivals~($\varepsilon$) and increasing service rate~($\tau$). \emph{Notes}: $\mu=1$, $\gamma=1$, $\lambda=R=200$, and $s/R=0.7, 0.8, 0.9, 1.$}}\label{fig:PQ-under-CBC}
\end{figure}

\textcolor{black}{We next fix the CBC scheme parameters ($\varepsilon$ and $\tau$) and examine the robustness of systems by changing their sizes. Specifically, we change two parameters, either~$s$ or~$\lambda$, and plot~$P_Q$ values; a flatter~$P_Q$ line implies that a system is more robust. We first alter the staffing level $s$ while fixing other parameters and observe the $P_Q$ plots. Figures~\ref{fig:Non-Asym-PQ-for-ErlangA} and~\ref{fig:Non-Asym-PQ-for-ErlangAE} correspond to the modified Erlang A models with no intervention ($\varepsilon=\tau=0$) and with intervention ($\varepsilon=0.1,\tau=0.05$), respectively. Notice that the asymptotic lines of $P_Q$ are different in these two figures: In Figure~\ref{fig:Non-Asym-PQ-for-ErlangA} ($R_Q=R$ case), $P_Q$ approaches a step function at $s=R$ as $R$ increases, indicating that the system operating in a QED regime is very sensitive to parameters, while in Figure~\ref{fig:Non-Asym-PQ-for-ErlangAE} ($R_Q<R$ case), $P_Q$ approaches a linear asymptotic line $P_Q=\frac{1-\frac{s}{R}}{1-\frac{R_Q}{R}}=\frac{1+\tau}{\varepsilon+\tau}\left(1-\frac{s}{R}\right)$ for $\frac{R_Q}{R}\le \frac{s}{R}\le 1$ as $R$ increases. These two different asymptotic lines for $P_Q$ are presented in Table~\ref{tab:asymptotic representation of Erlang A}. Note that in Table~\ref{tab:asymptotic representation of Erlang A}, only a square-root staffing representation (using~$c$) is presented for the $R=R_Q$ case; however, this square-root asymptotic line approaches a step function at $s=R$ in the limit of large $R$. The comparison of the two asymptotic lines indicates that the CBC scheme ($R_Q<R$) makes the system more robust to changes in parameters ($s$ for this case) in order to maintain $P_Q$ in a QED regime.}
\begin{figure}[H]
\centering
{\includegraphics*[scale=.45]{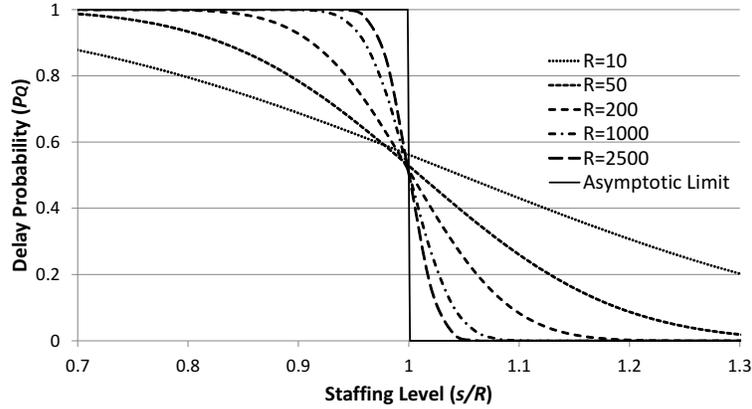}}
\caption{\textcolor{black}{Impact of change in $s/R$ on $P_Q$ (linear scale; $R=R_Q$ case). \emph{Notes}: $\mu=1$, $\gamma=1$, $\varepsilon=\tau=0$, and $\lambda=R=10, 50, 200, 1000, 2500$.}}\label{fig:Non-Asym-PQ-for-ErlangA}
\end{figure}
\begin{figure}[H]
\centering
{\includegraphics*[scale=.45]{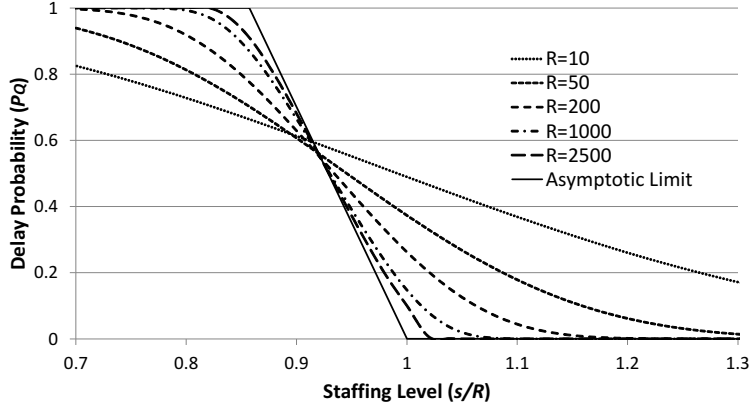}}
\caption{\textcolor{black}{Impact of change in $s/R$ on $P_Q$ (linear scale; $R_Q<R$ case). \emph{Notes}: $\mu=1$, $\gamma=1$, $\varepsilon=0.1$, $\tau=0.05$, and $\lambda=R=10, 50, 200, 1000, 2500$.}}\label{fig:Non-Asym-PQ-for-ErlangAE}
\end{figure}

\textcolor{black}{Finally, we observe the impact of traffic intensity $R/s\:(=\lambda/(s\mu))$ on $P_Q$ by changing $\lambda$ (while fixing $s$ and $\mu$). Figure~\ref{fig:Robustness-of-ErlangAE} compares no intervention ($R_Q=R$; $\varepsilon=\tau=0$) and intervention ($R_Q<R$; $\varepsilon=0$, $\tau=0.1, 0.2$) cases. The figure shows how $P_	Q$ responds to the change in $\lambda$ while fixing other parameters (i.e., we control $\lambda$ to vary $R/s$ to draw each line): If the $P_Q$ line is flatter, the system is more robust to the change in $\lambda$. Figure~\ref{fig:Robustness-of-ErlangAE} shows that smaller systems are proportionally less sensitive to changes in $\lambda$ (note that we fix $\mu$ in this figure), while for larger service systems, it is important to control (i.e., lower) $R_Q$, enabling the system to operate in a wider region of a QED regime. Asymptotic lines indicated in Figure~\ref{fig:Robustness-of-ErlangAE} are: a step function at $\frac{R}{s}=1$ for the $\tau=0$ case and close-to-linear lines $P_Q=\frac{1-\frac{s}{R}}{1-\frac{R_Q}{R}}=\frac{1+\tau}{\varepsilon+\tau}\frac{\frac{R}{s}-1}{\frac{R}{s}}\:(\approx \frac{1+\tau}{\varepsilon+\tau}(\frac{R}{s}-1))$ at $1\le \frac{R}{s}\le \frac{R}{R_Q}\:(=\frac{1+\tau}{1-\varepsilon})$ for the $\tau=0.1, 0.2$ cases (see Table~\ref{tab:asymptotic representation of Erlang A}).}
\begin{figure}[H]
\centering
{\includegraphics*[scale=0.45]{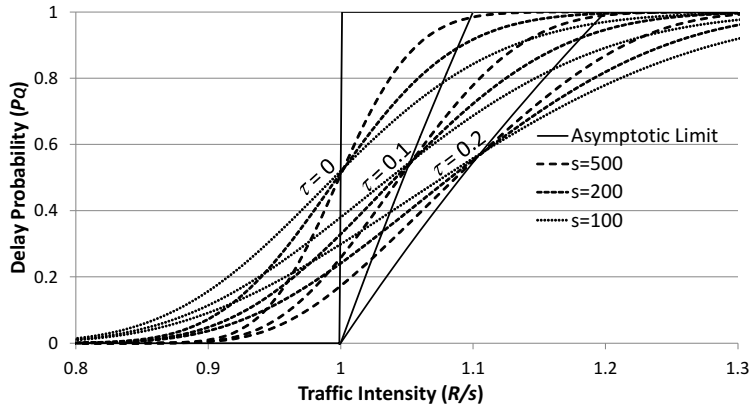}}
\caption{\textcolor{black}{Impact of change in $R/s$ on $P_Q$ (linear scale). \emph{Notes}: $\mu=1$, $\gamma=1$, $\varepsilon=0$, $\tau=0, 0.1, 0.2$, $R/s=\lambda/s$, and $s=100, 200, 500$.}}\label{fig:Robustness-of-ErlangAE}
\end{figure}


\section{Conclusion}\label{sec:conclusion}
\textcolor{black}{In this paper we study a modified Erlang~A model, which allows arrival and service rates to change when systems are congested---we call this intervention the congestion-based control (CBC) scheme. We derive non-asymptotic and asymptotic normal representations for the quality of service (QoS) performance indicators: Non-asymptotic formulae provide easy-to-calculate yet accurate results for systems of both small and large sizes, while asymptotic formulae provide insight into the QoS levels of very large systems. Specifically, by drawing the three distinct asymptotic regimes (ED, QD, and QED) in a phase diagram, we reveal that the highly sensitive QoS properties of systems with customer abandonment can be attributed to operation at the point of singularity where the three regimes co-exist. We demonstrate that the CBC scheme can avoid this point of singularity, and thus achieve a more robust system that is less sensitive to changes in parameters.}

\textcolor{black}{For the analysis of the modified Erlang A model under the CBC scheme, we utilize the Markov chain decomposition method \cite{sasanuma2019markov}. This method allows us to decompose a complex Markov chain into simpler sub-chains, which are then analyzed individually; this way, we can treat each sub-system precisely while still maintaining the key properties (such as the resource requirement) of each sub-system in the final representation of performance indicators. In fact, the explicit dependence of the asymptotic QoS indicators on each decomposed sub-chain is easy to derive thanks to the Markov chain decomposition method. We believe that our methodology to derive the results in this study can be instrumental in understanding the properties of other complicated Markov chain models.}

\clearpage



\bibliography{abandonment-ref} 






\newpage
\appendix
\appendixtitleon
\begin{appendices}{\textbf{\large Appendix}}

\section{Proofs}


\begin{proof}[Proof of Lemma \ref{str-rep}.]

We first prove Equations \eqref{pi_s} and \eqref{P_Q}. Note that from Corollary 1.10. in \cite{kelly1979reversibility2}, ${\pi_s^i=\text{Pr}\{\text{\# of customers}=s|\text{ system is in sub-chain }i\}=\pi_s/P_i}$, where the state $s$ is in sub-chain $i$. Therefore, $P_i=\pi_s/\pi_s^i$ holds for sub-chain~$i\:(=1,2)$. Using this result and the property that both sub-chains overlap at a single state $s$, we have
$$1=\sum _{i=1}^{2}P_{i}  -\pi _{s} =\sum _{i=1}^{2}\frac{\pi _{s} }{\pi _{s} ^{i} }  -\pi _{s} =\pi _{s} \cdot \left(\sum _{i=1}^{2}\frac{1}{\pi _{s} ^{i} }  -1\right),$$
from which we obtain Equations \eqref{pi_s} and \eqref{P_Q}.

In the derivation above, we use the fact from \cite{kelly1979reversibility2} that the sub-chain's steady-state probability is equal to the conditional steady-state probability of the full Markov chain (MC). However, this assumption is not trivial. For our model, this holds because the modified Erlang A MC is reversible. If it is not reversible, we need to decompose the MC more carefully to ensure the steady-state probabilities of sub-chains equal the conditional steady-state probabilities of the full chain.

We next derive Equations \eqref{P_AB} and \eqref{P_AB2}. From the flow balance condition, the total arrival rate $\lambda $ should be equal to the sum of the average number of customers being serviced per unit time and the average number of customers reneging/balking per unit time. Using this condition, the average number of customers being serviced given that a state is in sub-chain 2, is calculated as~${0\cdot \pi _{s}^{2} +s\mu _{Q} \cdot (1-\pi _{s}^{2})=s\mu _{Q} \cdot (1-\pi _{s}^{2})}$ because the departure rate $s\mu _Q$ only exists when the system (sub-chain 2) is busy. Therefore, the proportion of reneging/balking customers among the total arriving customers given in sub-chain 2 is
\footnotesize
$$\frac{\lambda -s\mu _Q \cdot (1-\pi _{s}^{2} )}{\lambda } = \pi _s^2 + \left(1-\frac{s\mu _Q}{\lambda}\right)(1 - \pi _s^2)=\frac{1+ (1-\frac{s\mu _Q}{\lambda})(\frac{1}{\pi_s^2} - 1)}{\frac{1}{\pi_s^2}}=\frac{1+ p\cdot (\frac{1}{\pi_s^2} - 1)}{\frac{1}{\pi_s^2}},$$
\normalsize
where $p=1-(s\mu_Q /\lambda )  = 1-(1 + \tau )(s\mu /\lambda ) = 1-(1 + \tau )(a+1)$. (Note that $s\mu/\lambda=s/R=a+1$ from the definition of the linear coefficient $a$.) This proportion is equivalent to the conditional probability of abandonment given in sub-chain 2:
$$P_{ab}^2=\frac{1+ p\cdot (\frac{1}{\pi_s^2} - 1)}{\frac{1}{\pi_s^2}}.$$
Hence, using the total probability theorem, Equations \eqref{P_AB} and \eqref{P_AB2} are proved as
$${P_{ab}} ={P_{ab}^2} \cdot {P_Q}= \frac{1+ p\cdot (\frac{1}{\pi_s^2} - 1)}{\frac{1}{\pi_s^2}}{P_Q}
= \frac{1+ p\cdot (\frac{1}{\pi_s^2} - 1)}{\frac{1}{\pi_s^1}+\frac{1}{\pi_s^2}-1}.$$

Next, we show Equations~\eqref{L_Q} and~\eqref{L_Q2}. Equation~\eqref{L_Q} is obtained from the flow balance requirement: ${\lambda P_{ab} =\gamma L_{Q} +\varepsilon \lambda P_{Q}}$ for reneging systems and $\lambda P_{ab} =\delta L_{Q} +\varepsilon \lambda P_{Q}$ for balking systems. (Note that $\delta \cdot L_{Q} =\delta \cdot \left(\sum _{k=0}^{u}k\pi _{s+k}  \right)=\sum _{k=0}^{u}\left(\delta k\right)\pi _{s+k}$ is the average number of balking customers, where we define~${u=\left\lfloor \lambda _{Q} /\delta \right\rfloor}$; at states greater than $s+u$ no customer enters the system and all arrivals balk.) Equation~\eqref{L_Q2} is obtained by plugging Equations~\eqref{P_Q} and~\eqref{P_AB2} into Equation~\eqref{L_Q}.
\end{proof}

\begin{remark}
The average delay time can also be represented by the blocking probabilities. For the reneging (or balking) system, we have the following expressions, respectively:
\begin{equation*}
\label{W_Q reneging}
W_{Q} =\frac{1}{\gamma } \cdot \frac{P_{ab} -\varepsilon P_{Q} }{1-\varepsilon P_{Q} }=\frac{1}{\gamma} \cdot \frac{(1-\varepsilon)+ (p-\varepsilon)(\frac{1}{\pi _s^2}-1)}{(\frac{1}{\pi _s^1}-\varepsilon)+(1-\varepsilon)(\frac{1}{\pi _s^2}-1)}
\end{equation*}
or
\begin{equation*}
\label{W_Q balking}
W_{Q} =\frac{1}{\delta } \cdot \frac{P_{ab} -\varepsilon P_{Q} }{1-P_{ab} }=\frac{1}{\delta} \cdot \frac{(1-\varepsilon)+ (p-\varepsilon)(\frac{1}{\pi _s^2}-1)}{(\frac{1}{\pi _s^1}-1)+(1-p)(\frac{1}{\pi _s^2}-1)}.
\end{equation*}
These formulae can be proved by Little's law, $L_{Q} =\lambda _{{\rm eff}} W_{Q} $, where the effective arrival rate for reneging systems and balking systems are  $\lambda _{{\rm eff}} =\lambda \cdot (1-P_{Q} )+\lambda _{Q} P_{Q} =\lambda \cdot (1-P_{Q} )+(1-\varepsilon )\lambda P_{Q} =\lambda \cdot (1-\varepsilon P_{Q} )$ and $\lambda _{{\rm eff}} =\lambda \cdot (1-P_{ab} )$, respectively.
\end{remark}

\
\begin{proof}[Proof of Corollary \ref{PQ-Pab}.]
We denote $\pi_s^1$ as $P_{block}$. Using Equation \eqref{P_Q}, we obtain $$\frac{1}{\pi _{s}^{2} } =\frac{P_{Q} }{1-P_{Q} } \cdot \frac{1-P_{block} }{P_{block} }$$ and equivalently $$\frac{1}{\pi _{s}^{2} } -1=\frac{P_{Q} -P_{block} }{\left(1-P_{Q} \right)P_{block} }. $$ By plugging these two equations into Equation \eqref{P_AB}, we obtain Equation \eqref{PQ-Pab-Pblock}.
\end{proof}
\
\begin{proof}[Proof of Corollary \ref{trade-off}.]
The first part of Corollary \ref{trade-off} may be obvious intuitively: $P_Q$ decreases as any of the abandonment-related parameters increase. To prove this result, note that $1/\pi_s^2=1+\sum\nolimits_{k=1}^\infty\prod\nolimits_{i=0}^{k-1}(\lambda_{s+i}/\mu_{s+i+1})$ (Equation (3.12) in \cite{kleinrock1975queueing2}), where $\lambda_k$ and $\mu_k$ for $k \geq s$ are defined in Section 3.1. An increase in one of the abandonment-related parameters ($\gamma$, $\delta$, and $\varepsilon$) always leads to a reduction in $\lambda_k$ or an increase in $\mu_k$, reducing $1/\pi_s^2$, while $(1/\pi_s^1)-1$ remains positive and constant. From Equation \eqref{P_Q}, we conclude that $P_Q$ decreases as any of the abandonment-related parameters increase. For the second part, assume that all performance-related parameters ($s$, $\lambda$, $\mu$, and $\tau$) are fixed, and hence, $p$ $(=1-s\mu_Q/\lambda)$ and $P_{block}$ are fixed as well. In this case, $P_Q$ and $P_{ab}$ are the only indicators that depend on abandonment-related parameters. If $p<P_{block}$ ($p>P_{block}$) holds, Equation \eqref{PQ-Pab-Pblock} shows there exists (does not exist) a trade-off between $P_Q$ and $P_{ab}$, respectively. Combining this result with the first part of this corollary, we can derive the second part.
\end{proof}

\begin{remark}
Examples of both $p<P_{block}$ and $p>P_{block}$ cases are easy to find: If $s\mu_Q$ is larger than $\lambda$, $p$ becomes negative and $p<P_{block}$ holds; in contrast, if $s\mu_Q$ is much smaller than $\lambda$, $p$ becomes close to 1 and $p>P_{block}$ holds. From our non-asymptotic approximation, we can see that non-negative $\tau$ is typically a sufficient condition for $p<P_{block}$ to hold, and also that negative $\tau$ is often a necessary condition for $p>P_{block}$ to hold.
\end{remark}

\
\begin{proof}[Proof of Lemma \ref{block-poisson}.]
We discuss each sub-chain separately. Define R.V.'s and parameters as in Table \ref{tab:setup-poisson}. Assume $s$, $s'$, and $s''$ are non-negative integers.

\begin{enumerate}

\item M/M/s/s sub-chain (left sub-chain): This result is shown in \cite{harchol2013performance2}.

For all $k=0,1,2,\dots,s $,
\begin{eqnarray}
\pi _k^1 &=& \pi _{k + 1}^1\frac{{(k + 1)\mu }}{\lambda } = \pi _{k + 1}^1\frac{{k + 1}}{R} = ... = \pi _s^1\frac{{(k + 1)(k + 2) \cdots s}}{{{R^{s - k}}}}  \nonumber \\
&=& \pi _s^1\frac{{s!}}{{k!}}\frac{{{R^k}}}{{{R^s}}}=\pi _s^1\frac{{{e^{ - R}}{R^k}/k!}}{{{e^{ - R}}{R^s}/s!}} = \pi _s^1\frac{{\Pr \{ {X_P} = k\} }}{{\Pr \{ {X_P} = s\} }}.\nonumber
\end{eqnarray}
By summing up the terms with respect to $k$ and applying the normalization condition, we obtain
$$\frac{1}{\pi_s^1} = \frac{\Pr \{X_P \le s\}}{\Pr \{X_P =s\}} =\frac{F_{P}(s;R)}{f_{P}(s;R)}.$$

We can confirm that this result matches the Erlang Loss (Erlang B) formula:
$$\frac{1}{{\pi _s^1}} = \frac{{\sum\limits_{k = 0}^s {{{(\lambda /\mu )}^i}/k!} }}{{{{(\lambda /\mu )}^s}/s!}}.$$

\item Reneging sub-chain (right sub-chain):

For all $k=0,1,2,\dots $,
\begin{eqnarray}
\pi _{s + k}^2 &=& \pi _{s + k - 1}^2\frac{{{\lambda _Q}}}{{s{\mu _Q} + k\gamma }} = \pi _{s + k - 1}^2\frac{{({\lambda _Q}/\gamma )}}{{(s{\mu _Q}/\gamma ) + k}} = \pi _{s + k - 1}^2\frac{{R'}}{{s' + k}}  \nonumber \\
&=& \pi _{s + k - 2}^2\frac{{R'}}{{s' + k}}\cdot \frac{{R'}}{{s' + k-1}}=... =\pi _s^2\frac{{R{'^k}}}{{(s' + 1)(s' + 2) \cdot  \cdot  \cdot (s' + k)}} \nonumber \\
&=& \pi _s^2\frac{{{e^{ - R'}}R{'^{s' + k}}/(s' + k)!}}{{{e^{ - R'}}R{'^{s'}}/s'!}} = \pi _s^2\frac{{\Pr \left\{ {X{'_P} = s' + k} \right\}}}{{\Pr \left\{ {X{'_P} = s'} \right\}}}. \nonumber
\end{eqnarray}
 
By summing up the terms with respect to $k$ and applying the normalization condition, we obtain
$$\frac{1}{\pi_s^2} = \frac{\Pr \{X'_P \ge s'\}}{\Pr \{X'_P =s'\}}.$$

We further rewrite this representation using the Poisson PMF/CDF:

\begin{eqnarray}
\frac{1}{\pi _{s}^{2} } &=& \frac{\Pr \left\{X'_{P} \ge s'\right\}}{\Pr \left\{X'_{P} =s' \right\}} =\frac{\Pr \left\{X'_{P} =s' \right\}}{\Pr \left\{X'_{P} =s' \right\}} +\frac{\Pr \left\{X'_{P} \ge s' +1\right\}}{\Pr \left\{X'_{P} =s' \right\}}  \nonumber \\
&=& 1+\frac{1-\Pr \left\{X'_{P} \le s' \right\}}{\Pr \left\{X'_{P} =s' \right\}}= 1+\frac{1-F_P(s';R')}{f_P(s';R')}. \nonumber
\end{eqnarray}

\item Balking sub-chain (alternate right sub-chain):

For all $k=0,1,2,\dots ,s''$,
\begin{eqnarray}
\pi _{s + k}^2 &=& \pi _{s + k - 1}^2\frac{{{\lambda _Q} - \left( {k - 1} \right)\delta }}{{s{\mu _Q}}} = \pi _{s + k - 1}^2\frac{{({\lambda _Q}/\delta ) - \left( {k - 1} \right)}}{{(s{\mu _Q}/\delta )}}\nonumber \\
&=& \pi _{s + k - 1}^2\frac{{s'' - \left( {k - 1} \right)}}{{R''}} = ...= \pi _s^2\frac{{(s'' - k + 1) \cdots (s'' - 1)s''}}{{R'{'^k}}}  \nonumber \\
&=& \pi _s^2\frac{{{e^{ - R''}}R'{'^{s'' - k}}/(s'' - k)!}}{{{e^{ - R''}}R'{'^{s''}}/s''!}} = \pi _s^2\frac{{\Pr \left\{ {X'{'_P} = s'' - k} \right\}}}{{\Pr \left\{ {X'{'_P} = s''} \right\}}}. \nonumber
\end{eqnarray}

By summing up the terms with respect to $k$ and applying the normalization condition, we obtain
$$\frac{1}{\pi_s^2} = \frac{\Pr \{X''_P \le s''\}}{\Pr \{X''_P =s''\}} = \frac{F_{P}(s'';R'')}{f_{P}(s'';R'')}.$$ 

\end{enumerate}
\end{proof}

\begin{proof}[Proof of Proposition \ref{Poisson-Normal}.]
\
\begin{enumerate}
\item Poisson CDF to standard normal CDF:

We first make a discrete-to-continuous conversion from Poisson to normal with a continuity correction ``+0.5". We then convert normal to standard normal using $c_{s;R}$ and $\Delta_R$:
\[{F_P}(s;R) \approx {F_N}(s + 0.5;R,\sqrt R ) = \Phi \left(\frac{(s+0.5)-R}{\sqrt{R}}\right)= \Phi \left(c_{s;R} + \Delta_R\right).\]

\item Poisson PMF to standard normal PDF:

Using the result above and the assumption that $\Delta_R$ is sufficiently small, we obtain
\small
\begin{eqnarray}
{f_P}(s;R) &=& {F_P}(s;R) - {F_P}(s-1;R) \approx {F_N}(s + 0.5;R,\sqrt R ) - {F_N}(s - 0.5;R,\sqrt R ) \nonumber \\
&\approx& \Phi (c_{s;R} + \Delta_R) - \Phi (c_{s;R} - \Delta_R) \approx 2\phi (c_{s;R} + \Delta_R)\Delta_R \nonumber \\
&=& \frac{\phi (c_{s;R} + \Delta_R)}{\sqrt R}=\frac{a_{s;R}\cdot \phi (c_{s;R} + \Delta_R)}{c_{s;R}}. \nonumber
\end{eqnarray}
\normalsize
\item Poisson modified hazard function to standard normal hazard function:

Using the above results and the definition of the hazard function for the standard normal distribution, it is straightforward to derive the following:
\small
\begin{equation}
\frac{f_P(s;R)}{1-F_P(s;R)} \approx \frac{\phi (c_{s;R}+\Delta_R)}{\sqrt{R} \cdot \left(1-\Phi (c_{s;R}+\Delta_R)\right)} =\frac{h(c_{s;R}+\Delta_R)}{\sqrt{R} }=\frac{a_{s;R}\cdot h(c_{s;R} + \Delta_R)}{c_{s;R}} \nonumber
\end{equation}
\begin{equation}
\frac{f_{P}(s;R)}{F_{P}(s;R)} \approx \frac{\phi (c_{s;R}+\Delta_R)}{\sqrt{R} \cdot \Phi (c_{s;R}+\Delta_R)} =\frac{h(-c_{s;R}-\Delta_R)}{\sqrt{R} }=\frac{a_{s;R}\cdot h(-c_{s;R} - \Delta_R)}{c_{s;R}}. \nonumber
\end{equation}
\normalsize
\end{enumerate}
\end{proof}
\

\begin{proof}[Proof of Proposition \ref{sq-root-staff}.]
Without calculation, the uniqueness and the existence of the solution can be inferred from the physical property of the model: If the number of staff increases from 0 to infinity, the exact as well as the non-asymptotic representation of $P_{Q}$ and $P_{ab}$ monotonically (strictly) decrease from 1 to 0. The non-asymptotic representation of $P_{Q}$ and $P_{ab}$ that satisfy such a property must have a unique solution to $P_{Q} =\alpha $ or $P_{ab} =\alpha $ for any $\alpha \in \left(0,1\right)$. The ceiling of the solution needs to be taken to obtain the optimal staffing level because a staffing level should be an integer whereas a solution (staffing coefficients $c$ or $a$) is not. 
\end{proof}

\section{Exact Representation of Performance Indicators for the Modified Erlang A Reneging Model}

The Poisson representation of performance indicators is exact but only for a set of parameters that satisfies integer constraints for staffing levels of the second (reneging/balking) sub-chain. If we want to know exact solutions for a general set of parameters, we should use the exact representation of the blocking probability of the sub-chain we are interested in. For a reneging sub-chain, the blocking probability is known as
\[\frac{1}{\pi _{s}^{2}} =\frac{s\mu _{Q}}{\gamma}\int _{0}^{1}e^{\lambda _{Q} t/\gamma } (1-t)^{\lambda _{Q} t/\gamma -1} dt. \] 

\begin{proof}
We follow \cite{coffman1994processor2}. Using gamma and beta functions, for $k=0,1,2,\dots $,
\footnotesize
\begin{eqnarray*}
\pi _{s+k}^{2} &=&\pi _{s+k-1}^{2} \cfrac{\lambda _{Q} }{s\mu _{Q} +k\gamma } =...=\pi _{s}^{2} (\lambda _{Q} /\gamma )^{k} (s\mu _{Q} /\gamma )\cfrac{1}{(s\mu _{Q} /\gamma )(s\mu _{Q} /\gamma +1)\cdot \cdot \cdot (s\mu _{Q} /\gamma +k)} \\
&=&\pi _{s}^{2} (s\mu _{Q} /\gamma )\cfrac{(\lambda _{Q} /\gamma )^{k} }{k!} \cfrac{\Gamma (s\mu _{Q} /\gamma )\Gamma (k+1)}{\Gamma (s\mu _{Q} /\gamma +k+1)} =\pi _{s}^{2} \cfrac{(\lambda _{Q} /\gamma )^{k} }{k!} B\left(s\mu _{Q} /\gamma ,k+1\right) \\
&=&\pi _{s}^{2} (s\mu _{Q} /\gamma )\cfrac{(\lambda _{Q} /\gamma )^{k} }{k!} \int _{0}^{1}t^{k} (1-t)^{s\mu _{Q} /\gamma -1} dt =\pi _{s}^{2} (s\mu _{Q} /\gamma )\int _{0}^{1}\cfrac{(\lambda _{Q} t/\gamma )^{k} }{k!} (1-t)^{s\mu _{Q} /\gamma -1} dt.
\end{eqnarray*}
\normalsize
From the normalization condition, 
\begin{eqnarray*}
1&=&\sum _{k=0}^{\infty }\pi _{s+k}^{2}  =\pi _{s}^{2} (s\mu _{Q} /\gamma )\int _{0}^{1}\left(\sum _{k=0}^{\infty }\cfrac{(\lambda _{Q} t/\gamma )^{k} e^{-\lambda _{Q} t/\gamma } }{k!}  \right)e^{\lambda _{Q} t/\gamma } (1-t)^{\lambda _{Q} t/\gamma -1} dt  \\
&=&\pi _{s}^{2} \cdot \frac{s\mu _{Q}}{\gamma} \int _{0}^{1}e^{\lambda _{Q} t/\gamma } (1-t)^{\lambda _{Q} t/\gamma -1} dt.
\end{eqnarray*}
\end{proof}

Together with the exact solution (Erlang Loss formula) for the left sub-chain (an M/M/s/s queue), which is 

$$\frac{1}{\pi _{s}^{1} } =\frac{\sum _{i=0}^{s}\left(\lambda /\mu \right)^{i} /i! }{\left(\lambda /\mu \right)^{s} /s!} =\frac{\sum _{i=0}^{s}e^{-\lambda /\mu } \left(\lambda /\mu \right)^{i} /i! }{e^{-\lambda /\mu } \left(\lambda /\mu \right)^{s} /s!} =\frac{\Pr \left\{X_{P} \le s\right\}}{\Pr \left\{X_{P} =s\right\}},$$
we can derive the exact representation of performance indicators from Lemma~\ref{str-rep}. However, the exact representation is analytically complex and harder to evaluate compared to the normal representation of performance indicators we derive in this paper.

\section{Application of Lemma \ref{str-rep}}
Lemma \ref{str-rep} holds regardless of the structure of the left sub-chain (sub-chain 1). For example, if the right sub-chain (sub-chain 2) is a single state $s$, then using the relationships $\pi _s^1=P_{block}$ (by definition) and $\pi _s^2=1$, we obtain an (obvious) general representation for the ``Erlang B" model (where the left sub-chain is not necessarily an M/M/s/s queue): $P_Q=P_{ab}=P_{block}$. Another example is that if the right sub-chain is an M/M/1 queue with system utilization $\rho \doteq \lambda/(s\mu)$, then using the relationships $\pi _s^1=P_{block}$ (by definition), $\pi _s^2=1-\rho$, $p=-a$, and $a=(s-R)/R=(1-\rho)/\rho$, we obtain a general representation for the ``Erlang C" model (where again, the left sub-chain is not necessarily an M/M/s/s queue): ${P_Q} = \dfrac{P_{block}}{1 - \rho +\rho P_{block}}$ and $P_{ab}=0$. These relationships can be confirmed for the standard Erlang C model by a direct calculation (for example, see Equation (15.5) in \cite{harchol2013performance2}).

Lemma \ref{str-rep} and Corollary \ref{str-rep-P_Q-} are both exact, and therefore can be utilized when deriving non-asymptotic formulae for performance indicators, but can also be utilized to obtain approximate results. For example, if $\pi_s$ is small, $\pi_s$  can be dropped from Equations \eqref{P_Q-P_Q-}, \eqref{P_ab-P_Q-}, and \eqref{L_Q-P_Q-} to obtain the following approximate expressions:

$$P_Q\approx P_{Q-},$$
$$P_{ab} \approx p P_{Q-},$$
$$L_Q\approx \dfrac{\lambda}{\theta} \cdot (p-\varepsilon) P_{Q-}= \dfrac{\lambda_Q-s\mu_Q}{\theta} P_{Q-}.$$

These approximate expressions are useful when staffing level $s$ is either in shortage or in excess enough to make $\pi_s$ close to 0. For example, if $s$ is in extreme shortage, $\pi_s^2$ gets close to 0 while $\pi_s^1$ does not, and therefore, we can approximate $\pi_s$ by 0 and $P_{Q-}$ by 1, in which case the above approximate expressions become a frequently used heavy-traffic approximation: $P_Q \approx 1$, $P_{ab} \approx p$, and $L_Q \approx \dfrac{\lambda_Q-s\mu_Q}{\theta}$.

\section{Comparison of Non-Asymptotic Representation and Square-root Staffing Rule for Erlang A model}

For simplicity, define $\phi(c)$ and $\omega(c)$ as follows (note: $\phi(c)$ is defined in Equation \eqref{phi}):
\begin{equation*}
\phi(c) \doteq \dfrac{\frac{\sqrt{\mu_Q/\theta}}{h \left(\sqrt{\mu_Q/\theta}\cdot c \right)}}{\frac{1}{h \left(-c\right)}+\frac{\sqrt{\mu_Q/\theta}}{h \left(\sqrt{\mu_Q/\theta}\cdot c \right)}} \text{ and }\omega(c) \doteq \dfrac{1}{\frac{1}{h \left(-c\right)}+\frac{\sqrt{\mu_Q/\theta}}{h \left(\sqrt{\mu_Q/\theta}\cdot c \right)}}.
\end{equation*}

Specifically, for the original Erlang A reneging model, we assume $\mu_Q=\mu$ and $\theta=\gamma$, and denote
\begin{equation*}
\phi^A(c) \doteq \dfrac{\frac{\sqrt{\mu/\gamma}}{h \left(\sqrt{\mu/\gamma}\cdot c \right)}}{\frac{1}{h \left(-c\right)}+\frac{\sqrt{\mu/\gamma}}{h \left(\sqrt{\mu/\gamma}\cdot c \right)}} \text{ and } \omega^A(c) \doteq \dfrac{1}{\frac{1}{h \left(-c\right)}+\frac{\sqrt{\mu/\gamma}}{h \left(\sqrt{\mu/\gamma}\cdot c \right)}}.
\end{equation*}
Denote also that $p^A \doteq -c/\sqrt{R}$, which is Equation \eqref{p} for the $\mu_Q=\mu$ case. Then the square-root staffing rule for the original Erlang A reneging model is represented as follows:
\begin{equation}
P_Q =\phi^A(c) \text{ and } P_{ab} = \frac{\omega^A(c)}{\sqrt{R}}+p^A \cdot \phi^A(c).
\label{square-root staffing for Erlang A}
\end{equation}

In contrast, by assuming $\varepsilon=\tau=0$ in Tables \ref{tab:normal-parameter} and \ref{tab:norm-rep}, a non-asymptotic representation of $P_{Q-}$ and $\pi_s$ becomes
\begin{equation*}
P_{Q-}^A \doteq \dfrac{\frac{\sqrt{\mu/\gamma}}{h \left(\sqrt{\mu/\gamma}\cdot c +\Delta' \right)}}{\frac{1}{h \left(-c-\Delta\right)}+\frac{\sqrt{\mu/\gamma}}{h \left(\sqrt{\mu/\gamma}\cdot c +\Delta' \right)}} \text{ and } \pi_s^A \doteq \dfrac{\dfrac{1}{\sqrt{R}}}{\frac{1}{h \left(-c-\Delta\right)}+\frac{\sqrt{\mu/\gamma}}{h \left(\sqrt{\mu/\gamma}\cdot c +\Delta' \right)}}.
\end{equation*}
Then using Corollary \ref{str-rep-P_Q-}, our non-asymptotic formulae for the original Erlang A reneging model become
\begin{equation}
P_Q = \pi_s^A+P_{Q-}^A \text{ and } P_{ab} =\pi_s^A+p^A \cdot P_{Q-}^A,
\label{non-asymptotic staffing for Erlang A}
\end{equation}

By comparing Equations \eqref{square-root staffing for Erlang A} and \eqref{non-asymptotic staffing for Erlang A}, we can see that the term corresponding to $\pi_s^A$ is missing from $P_Q$ of the square-root staffing rule. Note that the continuity correction terms ($\Delta$ and $\Delta'$) are also missing in the square-root staffing rule, but their contributions are smaller than $\pi_s^A$.

\begin{remark}
Using our non-asymptotic representation, we can derive a more general square-root staffing rule for the modified Erlang A reneging model with $R=R_Q$ (i.e., $\varepsilon+\tau=0$). By assuming $\varepsilon+\tau=0$ in Tables \ref{tab:normal-parameter} and \ref{tab:norm-rep} and ignoring all continuity correction terms, a non-asymptotic representation of $P_{Q-}$ and $\pi_s$ becomes
\begin{equation*}
P_{Q-} = \phi(c) \text{ and } \pi_s = \dfrac{\omega(c)}{\sqrt{R}}.
\end{equation*}
Denote that $p^* \doteq \varepsilon-(1-\varepsilon)c/\sqrt{R}$, which is Equation \eqref{p} for the $\varepsilon+\tau=0$ case. Then using Corollary \ref{str-rep-P_Q-}, the square-root staffing rule for the modified Erlang A reneging model with $R=R_Q$ is represented as follows:
\begin{equation}
P_Q =\phi(c)+\frac{\omega(c)}{\sqrt{R}} \text{ and } P_{ab} =\frac{\omega(c)}{\sqrt{R}}+p^* \cdot \phi(c).
\label{square-root staffing for modified Erlang A with R=R_Q}
\end{equation}
Equation \eqref{square-root staffing for modified Erlang A with R=R_Q} is more general and precise (for $P_Q$) than Equation \eqref{square-root staffing for Erlang A} (the original square-root staffing rule) but is not as accurate as our non-asymptotic representation because of missing the continuity correction terms.
\end{remark}







\end{appendices}
\end{document}